\documentclass[11pt]{amsart}
\usepackage{amsmath}
\usepackage{amsfonts}
\usepackage{latexsym}
\usepackage{amssymb}
\usepackage[dvips]{graphics}

\newcommand{\Z}{{\mathbb Z}}

\newcommand{\R}{{\mathbb R}}

\newcommand{\Sp}{{\mathbf S}}

\newcommand{\kk}{{\mathbf k}}
\newcommand{\cat}{{\rm {cat }}}
\newcommand{\Cat}{{\rm {Cat}}}

\newcommand{\Hom}{{\rm Hom}}
\newcommand{\im}{\rm im}

\newcommand{\comment}[1]{}

\def\ker{{\rm Ker }}
\def\im{{\rm Im }}

\def\diam{{\rm diam }}

\def\liminv{\lim\limits_\leftarrow}
\def\limone{{\lim\limits_\leftarrow}^1}

\newtheorem{theorem}{Theorem}%[section]
\newtheorem{proposition}{Proposition}
\newtheorem{lemma}[proposition]{Lemma}
\newtheorem{corollary}[proposition]{Corollary}

\newtheorem*{theorem*}{Theorem}
\theoremstyle{definition}
\newtheorem{definition}{Definition}
\newtheorem{example}{Example}
\newtheorem{remark}{Remark}
\begin{document}
\title{Closed 1-forms in topology and geometric group theory}

\author{Michael Farber, Ross Geoghegan and Dirk Sch\"utz}
\address{Department of Mathematics, University of Durham, Durham DH1 3LE, UK}
\email{michael.farber@durham.ac.uk}

\address{Department of Mathematics, SUNY Binghamton, NY 13902-6000, USA}
\email{ross@math.binghamton.edu}

\address{Department of Mathematics, University of Durham, Durham DH1 3LE, UK}
 \email{dirk.schuetz@durham.ac.uk}

%\thanks{}

%\subjclass[2000]{Primary 55N25; Secondary 55U99}

\date{\today}

%\keywords{}

\begin{abstract}
In this article we describe relations of the topology of closed 1-forms to the group theoretic invariants of Bieri-Neumann-Strebel-Renz. Starting with a survey, we extend these Sigma invariants to finite CW-complexes and show that many properties of the group theoretic version have analogous statements. In particular we show the relation between Sigma invariants and finiteness properties of certain infinite covering spaces. We also discuss applications of these invariants to the Lusternik-Schnirelmann category of a closed 1-form and to the existence of a non-singular closed 1-form in a given cohomology class on a high-dimensional closed manifold.
\end{abstract}

\maketitle

\begin{center}
{\it To S.P. Novikov on the occasion of his 70-th birthday}
\end{center}
\vskip 1cm

\section*{Introduction}
The last three decades have seen a growing interest in the topology of closed 1-forms ever since S.P. Novikov \cite{N1, noviko} introduced Morse theoretic techniques to study classical problems in mathematical physics. In analogy to the Morse-Smale complex of an ordinary Morse function on a closed manifold, he constructed a chain complex, now called the Novikov complex, associated to a closed 1-form whose singularities are non-degenerate in the sense of Morse. While Novikov's interest was to study such problems as Kirchhoff type equations \cite{noviko,N4,NSm}, other applications in different areas of mathematics would soon become apparent. Gradient vector fields of closed 1-forms, for example, give rise to fascinating results in dynamical systems, and some of the topics arising this way have been covered in a recent monograph \cite{farbook} and survey article \cite{fasc}.

Another quite surprising link has been made to geometric group theory, when it became clear that the work of Sikorav \cite{sikora}, who applied Novikov theory to symplectic topology, was closely related to the work of Bieri, Neumann, Strebel and Renz \cite{binest,bieren}. More specifically, the geometric invariants $\Sigma^k(G)$ of a group $G$, which contain important information on the finiteness properties of certain subgroups of $G$ and whose definition is recalled in Section \ref{basic}, can be described in terms of vanishing results of a generalized Novikov homology.

The feature that combines these areas is that closed 1-forms represent cohomology classes $\xi\in H^1(X;\R)$. If $X$ is a smooth manifold this is a special case of de Rham theory, and the first-named author has developed a theory of closed 1-forms on topological spaces for which this result holds more generally, see \cite{farber}. Now if $X$ is connected, $H^1(X;\R)$ can be identified with the set of homomorphisms ${\rm Hom}(\pi_1(X),\R)$, where $\R$ is considered as a group with the usual addition. Indeed, the invariants $\Sigma^k(G)$ can be viewed as subsets of the unit sphere in ${\rm Hom}(G,\R)$.

One purpose of this article is to describe these relations in a general setting, and to develop the theory of Bieri-Neumann-Strebel-Renz using the language of the topology of closed 1-forms. In particular, we extend the notion of Sigma invariants to finite CW-complexes. These invariants $\Sigma^k(X)$ for $k\geq 1$ are defined as generalizations of the group theoretic versions, and they have similar properties. This generalization is motivated by the fact that the group theoretic invariants regularly occur in the topology of closed 1-forms. For example, the condition $\xi\in \Sigma^2(\pi_1(M))$ appears implicitly in the work of Latour \cite{latour} as a necessary condition for the existence of a non-singular closed 1-form in $\xi\in H^1(M;\R)$, where $M$ is a high-dimensional closed manifold ($\dim M\geq 6$). Another necessary condition of \cite{latour}, which in fact implies $\xi\in \Sigma^2(\pi_1(M))$, is the contractibility of a certain function space. It turns out that contractibility of this space is equivalent to $\xi\in \Sigma^k(M)$ for all $k\geq 1$. We remark that for compact 3-manifolds the condition $\xi\in \Sigma^1(\pi_1(M))$ is sufficient for the existence of a non-singular closed 1-form representing $\xi$, see \cite{binest}.

Another property of these new Sigma invariants is that they reflect finiteness properties of infinite abelian covering spaces $q:\overline{X}\to X$. By an abelian covering we mean a regular covering with $\pi_1(X)/\pi_1(\overline{X})$ an abelian group. Denote
\begin{eqnarray*}
S(X,\overline{X})&=&\left\{0\not=\xi:\pi_1(X)\to\R\,\left|\,\xi|_{\pi_1(\overline{X})}=0\right.\right\}.
\end{eqnarray*}

\begin{theorem*}\label{maintheorem}
Let $X$ be a finite connected CW-complex and $q:\overline{X}\to X$ a regular covering  space with $\pi_1(X)/\pi_1(\overline{X})$ abelian. For $k\geq 1$, the following properties are equivalent.
\begin{enumerate}
\item $\overline{X}$ is homotopy equivalent to a CW-complex with finite $k$-skeleton.
\item $S(X,\overline{X})\subset \Sigma^k(X)$.
\end{enumerate}
Furthermore, if $S(X,\overline{X})\subset\Sigma^{\dim X}(X)$, then $\overline{X}$ is finitely dominated.
\end{theorem*}

The proof of this theorem is given in Section \ref{mainth}.

A common feature in the definitions and techniques is the notion of movability of subsets of a space $X$. Here movability is meant with respect to a closed 1-form $\omega$ representing a cohomology class $\xi\in H^1(X;\R)$. Roughly, movability of a set $A\subset X$ means that there is a homotopy $H$ of $A$ into $X$ starting with the inclusion and such that for every point $a\in A$ the integral of $\omega$ along the path $t\mapsto H_t(a)$ is large. While homological versions, using chain homotopies, of this appeared already in Bieri and Renz \cite[Thm.C]{bieren}, a topological version was formulated in a quite different context in developing a Lusternik-Schnirelmann theory for closed 1-forms, see \cite{farber,farbe4,farbook,farkap,fasc}. Due to the similar nature one expects a closer relation, which we derive in Section \ref{category}. A movability notion for homology classes is developed in Section \ref{movhomnotion}, which has applications to cup-length estimates for the Lusternik-Schnirelmann theory of a closed 1-form.

This article is written as a companion to the recent survey article \cite{fasc} which was focussing on applications of closed 1-forms in dynamical systems. The present paper contains a significant amount of new material, although parts of it are also meant as a survey.

\section{Bieri-Neumann-Strebel-Renz Invariants}\label{basic}
Let $G$ be a finitely generated group. We want to recall the definition of the Bieri-Neumann-Strebel-Renz invariants $\Sigma^k(G;\Z)$, introduced in \cite{binest,bieren}. We denote
\begin{eqnarray*}
S(G)&=& (\Hom(G,\R)-\{0\})/\R_+,
\end{eqnarray*}
that is, we identify nonzero homomorphisms, if one is a positive multiple of the other. This is a sphere of dimension $r-1$, where $r$ denotes the rank of the abelianization of $G$. We will identify $S(G)$ with the unit sphere in $\Hom(G,\R)$ (after choosing an inner product on the latter) and simply write $\xi\in S(G)$.

Given $\xi\in S(G)$, we denote
\begin{eqnarray*}
\Z G_\xi&=&\left\{\left.\sum_{g\in G} n_gg\in \Z G \,\right|\, n_g=0 \mbox{ for }\xi(g)<0\right\},
\end{eqnarray*}
a subring of $\Z G$.

We say that the trivial $\Z G$-module $\Z$ is of type $FP_k$ over $\Z G_\xi$, if there exists a resolution
\begin{equation}\label{resol}
 \ldots \longrightarrow F_i\longrightarrow F_{i-1} \longrightarrow \ldots \longrightarrow F_0\longrightarrow \mathbb{Z}\longrightarrow 0
\end{equation}
of $\Z$ by free $\Z G_\xi$-modules with each $F_i$ finitely generated for $i\leq k$.

\begin{definition}
 The \em Bieri-Neumann-Strebel-Renz invariants \em are now defined as
\begin{eqnarray*}
 \Sigma^k(G;\Z)&=&\{\xi\in S(G)\,|\,\Z \mbox{ is of type }FP_k\mbox{ over }\Z G_\xi\}.
\end{eqnarray*}
\end{definition}

The power of these invariants lies in the fact that they are closely related to the finiteness properties of subgroups of $G$. Let us recall the relevant finiteness properties.

\begin{definition}
For $k\geq 1$ a group $G$ is of type $FP_k$, if there is a resolution (\ref{resol}) of $\Z$ by free $\Z G$-modules with each $F_i$ finitely generated for $i\leq k$. Also, we say that $G$ is of type $F_k$, if there is an Eilenberg-MacLane space for $G$ with finite $n$-skeleton.
\end{definition}

We get that $G$ of type $F_k$ implies type $FP_k$ by looking at the cellular chain complex of the universal cover of the Eilenberg-MacLane space, and type $FP_1$ is equivalent to type $F_1$ which simply means finitely generated. But type $FP_2$ does not imply finitely presented, as the examples of Bestvina and Brady show \cite{besbra}. For more information on these finiteness properties see Brown \cite{brown} and Geoghegan \cite{geoghe}.

Notice that $\Z G$, when viewed as a $\Z G_\xi$-module for any $\xi$, is a direct limit of free $\Z G_\xi$-modules. It is therefore a flat $\Z G_\xi$-module. Furthermore, for every $\Z G$-module $A$ we have $\Z G\otimes_{\Z G_\xi}A\cong A$. Thus if $\Sigma^k(G;\Z)\not=\emptyset$ for some $k\geq 1$, we can apply $\Z G\otimes_{\Z G_\xi}-$ to the resolution (\ref{resol}) for some $\xi\in \Sigma^k(G;\Z)$, to get that $G$ is itself of type $FP_k$.

The following theorem, which we generalize in Section \ref{mainth}, relates finiteness properties of certain subgroups to the invariants.

\begin{theorem}[Bieri-Renz, \cite{bieren}]\label{bierenth}
 Let $G$ be a group of type $FP_k$, $N$ a subgroup of $G$ containing the commutator subgroup of $G$. Then $N$ is of type $FP_k$ if and only if $\Sigma^k(G;\Z)$ contains the subsphere $S(G,N)=\{\xi\in S(G)\,|\, N\leq \ker \,\xi\}$.
\end{theorem}

There also exists a version of Theorem \ref{bierenth} which gives a criterion for $N$ to be of type $F_k$, involving a homotopical invariant $\Sigma^k(G)$. We will see more about this invariant below. Another result, proven in \cite{bieren} is that all $\Sigma^k(G;\Z)$ are open subsets of $S(G)$.

In \cite{binest} it was shown that even the particular case $k=1$ has very important applications to group theory.

\begin{theorem}[Bieri-Neumann-Strebel, \cite{binest}]
 Let $G$ be a finitely presented group without non-abelian free subgroups. Then
\begin{eqnarray*}
 \Sigma^1(G;\Z)\cup -\Sigma^1(G;\Z)&=&S(G).
\end{eqnarray*}
\end{theorem}

Here $-\Sigma^1(G)$ denotes the image of $\Sigma^1(G)$ under the antipodal map.

We now want to give a more geometrical interpretation of these invariants, for the moment we will confine ourselves to the case $k=1$ and for simplicity we will assume that $G$ is finitely presented. Let $X$ be a finite CW-complex with $\pi_1(X)\cong G$ and let $q:\overline{X}\to X$ be the universal abelian covering. Given a non-zero homomorphism $\xi:G\to \R$ we can build a map $h:\overline{X}\to \R$ with $h(gx)=\xi(g)+h(x)$ for all $g\in G$ and $x\in \overline{X}$ by induction over the skeleta of $\overline{X}$. Note that $G$ acts on $\overline{X}$ by covering transformations, with the commutator subgroup acting trivially. Write $N=h^{-1}([0,\infty))$, then $N$ need not be connected, but has a unique component on which $h$ is unbounded, see \cite[Lm.5.2]{binest}. In the result below we assume that the basepoint of $N$ is chosen in this component.

\begin{theorem}[Bieri-Neumann-Strebel, \cite{binest}]\label{definv}
We have $\xi\in\Sigma^1(G;\Z)$ if and only if $i_\#:\pi_1(N)\to\pi_1(\overline{X})$ is an epimorphism, where $i:N\to\overline{X}$ is the inclusion.
\end{theorem}

This geometric criterion is closely related to a concept of movability of a subset of $X$ with respect to a given $\xi$ and which has recently been studied in connection with a Lusternik-Schnirelmann theory of such $\xi$, see \cite{farber,farbook,fasc}. Let us recall the definition of a closed 1-form on a topological space $X$.

\begin{definition}
 A \em continuous closed 1-form $\omega$ on a topological space $X$ \em is defined as a collection $\{f_U\}_{U\in\mathcal{U}}$ of continuous real-valued functions $f_U:U\to \R$ where $\mathcal{U}=\{U\}$ is an open cover of $X$ such that for any pair $U,V\in\mathcal{U}$ the difference
\[
 f_U|_{U\cap V}-f_V|_{U\cap V}:U\cap V\to \R
\]
is a locally constant function. Another such collection $\{g_V\}_{V\in\mathcal{V}}$ (where $\mathcal{V}$ is another open over of $X$) defines \em an equivalent \em closed 1-form if the union collection $\{f_U,g_V\}_{U\in\mathcal{U},V\in\mathcal{V}}$ is a closed 1-form, i.e., if for any $U\in\mathcal{U}$ and $V\in \mathcal{V}$ the function $f_U-g_V$ is locally constant on $U\cap V$.
\end{definition}

These closed 1-forms behave in the same way as smooth closed 1-forms on manifolds; they can be integrated along paths $\gamma:[a,b]\to X$, and integration along loops defines a homomorphism $\xi_\omega:\pi_1(X)\to\R$. Furthermore, every such homomorphism can be realized by a closed 1-form. See \cite[\S 3]{fasc} for details.

\begin{example}
A continuous function $f:X\to S^1$ determines a closed 1-form in the following way. Think of $S^1$ as $\R/\Z$ and let $p:\R\to\R/\Z$ be the projection. If $I=(a,b)\subset \R$ is an open interval with $b-a\leq 1$, then $I$ is homeomorphic to the open subset $p(I)\subset \R/\Z$ via $p$. The collection 
\begin{eqnarray*}
\omega&=&\left\{(p|_I)^{-1}\circ f|_{f^{-1}(p(I))}:f^{-1}(p(I))\to I\right\}
\end{eqnarray*}
then defines a closed 1-form. Furthermore, $\xi_{\omega}:\pi_1(X)\to\R$ can be identified with $f_\#:\pi_1(X)\to \pi_1(S^1)$, if the standard generator of $\pi_1(S^1)$ is identified with $1\in \R$.
\end{example}

\begin{definition}
 Let $X$ be a finite connected CW-complex, $G=\pi_1(X)$ and $\omega$ a closed 1-form on $X$. A subset $A\subset X$ is called \em $n$-movable with respect to $\omega$ and control $C\geq0$ \em (where $n$ is an integer), if there is a homotopy $H:A\times [0,1]\to X$ such that $H(a,0)=a$ or all $a\in A$, and
\begin{eqnarray*}
 \int_a^{H_1(a)}\omega &\geq & n
\end{eqnarray*}
and
\begin{eqnarray*}
 \int_a^{H_t(a)}\omega &\leq&-C
\end{eqnarray*}
for all $a\in A$ and $t\in[0,1]$. Here the integral is taken over the path $s\mapsto H(a,s)$ for $s\in [0,t]$.
\end{definition}

This notion of movability has its roots in the Lusternik-Schnirelmann theory of a closed 1-form, compare \cite{farber,fasc}. The next Proposition shows that it also gives a criterion for $\Sigma^1(G;\Z)$.

\begin{proposition}\label{movewithcontrol}
 Let $X$ be a finite connected CW-complex, $G=\pi_1(X)$ and $\xi:G\to \R$ a homomorphism which is represented by a closed 1-form $\omega$. Then the following are equivalent.
\begin{enumerate}
\item $\xi\in\Sigma^1(G;\Z)$.
\item There is a $C\geq 0$ such that the 1-skeleton $X^{(1)}\subset X$ is $n$-movable with respect to $\omega$ and control $C$ for every $n>0$.
\end{enumerate}
\end{proposition}

\begin{proof}
 (1) $\Longrightarrow$ (2): Let $\overline{X}$ be the universal abelian cover of $X$ and $h:\overline{X}\to\R$ be obtained from the pullback of $\omega$ to $\overline{X}$. It is easy to see that we have $h(gx)=\xi(g)+h(x)$ for all $x\in \overline{X}$ and $g\in G$. We choose $N=h^{-1}([0,\infty))$, and let $N'\subset N$ be the component such that $h$ is unbounded on $N'$ by \cite[Lemma 5.2]{binest}. For every cell $\sigma$ of $X$ pick a lift $\bar{\sigma}\subset h^{-1}((-\infty,-1])$. If $\sigma$ is a 0-cell, we can find a cellular map $H_\sigma:[0,1]\to\overline{X}$ such that $H_\sigma(0)=\bar{\sigma}$ and $H_\sigma(1)\in N'$. Here $[0,1]$ has the standard cell structure with two 0-cells and one 1-cell. Using equivariance, this gives an equivariant cellular homotopy $H^0:\overline{X}^{(0)}\times [0,1]\to\overline{X}$. Note that $G$ acts on $\overline{X}\times [0,1]$ by $g(x,t)=(gx,t)$, and $H^0$ induces a homotopy on $X^{(0)}$ which gives $1$-movability of the 0-skeleton. As the image of $H^0_1$ is in the 0-skeleton, we can iterate this homotopy to obtain $n$-movability for any $n>0$.

Now pick a cell $\bar{\sigma}\subset h^{-1}((-\infty,-1])$ for every 1-cell $\sigma$ of $X$ and let $u,v$ be the boundary points of $\bar{\sigma}$. By possibly iterating $H^0$, we can assume that $H^0(u,1),H^0(v,)\in N'$ and since $N'$ is connected, we can find a cellular map $H_{u,v}:[0,1]\to N'$ connecting these points.

Note that $H_{u,v}([0,1])$, $H^0(\{u,v\}\times [0,1])$ and $\bar{\sigma}$ combine to a closed loop in $\tilde{X}$ when suitably oriented. But by Theorem \ref{definv} this loop is representable by a loop in $N'$. In other words, by changing the path $H_{u,v}$ suitably in $N'$, we can assume that this loop bounds. Therefore we can extend $H^0$ to $H_\sigma:(\overline{X}^{(0)}\cup\bar{\sigma})\times [0,1]\to \overline{X}$ cellularly such that $H_\sigma(\bar{\sigma},1)=H_{u,v}([0,1])$. Doing this for every 1-cell of $X$ and extending equivariantly gives a cellular and equivariant homotopy $H^1:\overline{X}^{(1)}\times [0,1]\to\overline{X}$ such that $H^1_0$ is inclusion and $h(H^1(x,1))-h(x)\geq 1$.

Since $H^1$ is equivariant we get a homotopy $H:X^{(1)}\times [0,1]\to X$ which shows that $X^{(1)}$ is 1-movable with respect to $\omega$. By compactness there is a $C>0$ such that $X^{(1)}$ is 1-movable with control $C$. Again the image of $H_1$ is in the 1-skeleton of $X$, so we can iterate the homotopy. Note that iterating the homotopy does not increase the control, so we get that $X^{(1)}$ is $n$-movable with respect to $\omega$ with control $C$ for every $n>0$.

(2) $\Longrightarrow$ (1): Let $h:\overline{X}\to\R$ and $N'\subset N=h^{-1}([0,\infty))$ be as above. Pick a basepoint $x_0\in N'$ with $f(x_0)\geq C+1$. Let $\gamma:(S^1,1)\to (\tilde{X},x_0)$ be a loop which can be assumed cellular. By compactness of $S^1$ there is a $K\leq 0$ such that $\gamma(S^1)\subset f^{-1}([K,\infty))$.

By assumption there is a homotopy $\bar{H}:\overline{X}^{(1)}\times [0,1]\to\overline{X}$ with $\bar{H}_0$ is inclusion and
\begin{eqnarray*}
h(\bar{H}_1(x))-h(x)&\geq &C-K\\
h(\bar{H}_t(x))-h(x)&\geq &-C
\end{eqnarray*}
for all $x\in \overline{X}$. Let $\mu:\tilde{X}\to[0,1]$ be a map with $\mu|f^{-1}([C+1,\infty))\equiv 0$ and $\mu|f^{-1}((-\infty,C])\equiv 1$. Now define $A:S^1\times[0,1]\to\tilde{X}$ by $A(x,t)=\bar{H}(\gamma(x),\mu(\gamma(x))\cdot t)$. Then $A(x,0)=\gamma(x)$, $A(x,1)=\bar{H}(\gamma(x),\mu(\gamma(x)))\in N'$ and $A(x_0,t)=x_0$ for all $t\in[0,1]$. Therefore $\xi\in\Sigma^1(G;\Z)$ by Theorem \ref{definv}.
\end{proof}

A criterion analogous to condition (2) of Proposition \ref{movewithcontrol} leads to the homotopical version of the Bieri-Neumann-Strebel-Renz invariants $\Sigma^k(G)$, introduced in \cite{bieren}. For this we assume that $X$ is an Eilenberg-MacLane space with finite $n$-skeleton for some $n\geq 1$.

\begin{definition}\label{sigmaG}
Let $X$ be as above, $\xi\in S(G)$, $\omega$ a closed 1-form on $X$ representing $\xi$ and $k\geq 0$. We say that $\xi\in \Sigma^k(G)$, if there is $\varepsilon>0$ and a cellular homotopy $H:X^{(k)}\times [0,1]\to X$ such that $H(x,0)=x$ for all $x\in X^{(k)}$ and
\begin{eqnarray*}
\int_{\gamma_x}\omega & \geq & \varepsilon
\end{eqnarray*}
for all $x\in X^{(k)}$, where $\gamma_x:[0,1]\to X$ is given by $\gamma_x(t)=H(x,t)$. Here $X^{(k)}$ denotes the $k$-skeleton of $X$.
\end{definition}

The condition that $H$ is cellular ensures that $H_1$ has image in $X^{(k)}$, so that the homotopy can be iterated. As a result we see that $\varepsilon$ can be arbitrarily large which shows that the definition does not depend on the particular $\omega$. These iterations all have the same control $C\geq 0$. Using cellular approximations it is easy to see that condition (2) of Proposition \ref{movewithcontrol} is equivalent to $\xi\in \Sigma^1(G)$.

The above definition is not the usual definition of $\Sigma^k(G)$, but it follows from Proposition \ref{equivdef} below that it agrees with the definition given in Bieri and Renz \cite[\S 6]{bieren}. 

Even though it is generally quite difficult to describe $\Sigma^k(G)$ and $\Sigma^k(G;\Z)$, there are some important classes of groups for which the Sigma invariants can be determined, for example right-angled Artin groups, see \cite{memewy}, and Thompson's group $F$, see \cite{bigeko}. For more information and applications of these invariants see, for example, \cite{bieri,biegeo,binest,bieren,harkoc}.

\section{Sigma invariants for CW-complexes}\label{sigmacw}

There is no particular reason for $X$ to be aspherical in Definition \ref{sigmaG}, so we can extend this definition to more general $X$. For simplicity we will assume that $X$ is a finite connected CW-complex, but it is possible to consider the case where $X$ has finite $n$-skeleton for some $n\geq 1$, in which case we can define $\Sigma^k(X)$ for $k\leq n$. Let us first define
\begin{eqnarray*}
 S(X)&=&(\Hom(H_1(X),\R)-\{0\})/\R_+
\end{eqnarray*}
where we again identify a homomorphism with its positive multiples. Clearly $S(X)=S(G)$ with $G=\pi_1(X)$.

The \em Sigma invariants for CW-complexes \em are now defined in analogy to Definition \ref{sigmaG}.

\begin{definition}\label{sigmaX}
Let $X$ be a finite connected CW-complex, $\xi\in S(X)$, $\omega$ a closed 1-form on $X$ representing $\xi$ and $k\geq 0$. We say that $\xi\in \Sigma^k(X)$, if there is $\varepsilon>0$ and a cellular homotopy $H:X^{(k)}\times [0,1]\to X$ such that $H(x,0)=x$ for all $x\in X^{(k)}$ and
\begin{eqnarray*}
\int_{\gamma_x}\omega & \geq & \varepsilon
\end{eqnarray*}
for all $x\in X^{(k)}$, where $\gamma_x:[0,1]\to X$ is given by $\gamma_x(t)=H(x,t)$.
\end{definition}

Let $p:\tilde{X}\to X$ be the universal covering space. Just as in the smooth manifold case, a closed 1-form $\omega$ pulls back to an exact form on $\tilde{X}$, $p^\ast\omega=dh$ for some $h:\tilde{X}\to\R$ with
\begin{eqnarray}\label{equivheight}
h(gx)&=&h(x)+\xi(g)
\end{eqnarray}
for all $g\in\pi_1(X)$ and $x\in\tilde{X}$. A function $h:\tilde{X}\to \R$ with property (\ref{equivheight}) is called a \em height function\em. Such a height function defines a closed 1-form representing $\xi$.
A subset $N\subset \tilde{X}$ is called a \em neighborhood of $\infty$ with respect to $\xi$\em, if there exists a height function $h_\xi:\tilde{X}\to\R$ and $a \in \R$ such that
\begin{eqnarray*}
h_\xi^{-1}([a,\infty))&\subset &N.
\end{eqnarray*}
It is easy to check that if $N$ is a neighborhood of $\infty$ with respect to $\xi$ for some height function, it is also a neighborhood of $\infty$ with respect to $\xi$ for every other height function.

If a particular height function $h_\xi$ is given, we write $N_i=h_\xi^{-1}([i,\infty))$ for every $i\in \R$.

We can describe $\Sigma^k(X)$ in terms of height functions alone. For this we need one more definition.

\begin{definition}\label{pathtoinfinity}
Let $X$ be a finite connected CW-complex, $h_\xi:\tilde{X}\to\R$ a height function and $\xi\in H^1(X;\R)$ be nonzero. A \em path to $\infty$ with respect to $\xi$ \em is a map $\gamma:[0,\infty)\to \tilde{X}$ such that for every neighborhood $N$ of $\infty$ there is a $T\geq 0$ such that $\gamma(t)\in N$ for all $t\geq T$.
\end{definition}

Given a path $\gamma$ to $\infty$ we can pick points $\gamma(T_N)\in N$ for every neighborhood $N$ of $\infty$ and get an inverse system $\{\pi_\ast(\tilde{X},N,\gamma(T_N))\}$ where the basepoint change is done via $\gamma|[T_N,T_{N'}]$. We will often suppress the basepoints but we want to point out that there is always a path to $\infty$ in the background. The next proposition shows that the basepath is not important for our purposes.

\begin{proposition}\label{equivdef}
Let $X$ be a finite connected CW-complex, $\xi\in H^1(X;\R)$ be nonzero and $h_\xi:\tilde{X}\to\R$ a height function. The following are equivalent.
\begin{enumerate}
\item $\xi\in \Sigma^k(X)$.
\item There is a $\lambda\geq 0$ such that $j_\#:\pi_l(\tilde{X},N_i)\to \pi_l(\tilde{X},N_{i-\lambda})$ is trivial for all $l\leq k$ and every $i\in\R$.
\item For every neighborhood $N$ of $\infty$ with respect to $\xi$, there is another neighborhood $N'\subset N$ such that $j_\#:\pi_l(\tilde{X},N')\to \pi_l(\tilde{X},N)$ is trivial for all $l\leq k$.
\item There is an $\varepsilon >0$ and an equivariant cellular homotopy $\tilde{H}:\tilde{X}^{(k)}\times [0,1]\to \tilde{X}$, such that $\tilde{H}_0$ is inclusion and $h_\xi(\tilde{H}_1(\tilde{x}))-h_\xi(\tilde{x})\geq \varepsilon$ for all $\tilde{x}\in\tilde{X}^{(k)}$.
\end{enumerate}
\end{proposition}

\begin{proof}
(1) $\Longleftrightarrow$ (4) as we can lift $H$ to $\tilde{H}$ and $\tilde{H}$ determines $H$.
(2) $\Longrightarrow$ (3) is obvious.
(3) $\Longrightarrow$ (4): We define $H$ by induction on the skeleta of $\tilde{X}$. The homotopy can always be defined on $\tilde{X}^{(0)}$ as $\tilde{X}$ is connected and $\xi$ nonzero. Assume that $H:\tilde{X}^{(k-1)}\times [0,1]\to \tilde{X}$ satisfies the conclusion of (4). Let $N=N_0=h_\xi^{-1}([0,\infty))$ be a neighborhood of $\infty$ with respect to $\xi$ and $N'\subset N$ as in (3), and choose a $\varepsilon>0$. Choose a lift $\tilde{\sigma}\subset \tilde{X}-N_{-\varepsilon}$ for every $k$-cell $\sigma$ of $X$. By iterating $H$ if necessary, we can assume that $H_1(\partial \tilde{\sigma})\subset N'$. Given a characteristic map $\chi_\sigma:(D^k,S^{k-1})\to (\tilde{X}^{(k)},\tilde{X}^{(k-1)})$, we can compose $\chi_\sigma|S^{k-1}$ with $H$ to get an element of $\pi_k(\tilde{X},N')$ which restricts to the trivial element of $\pi_k(\tilde{X},N)$. This gives a homotopy $\chi:D^k\times[0,1]\to\tilde{X}$ with $\chi_0=\chi_\sigma$ and $\chi_1(D^k)\subset N$. We can use this homotopy to extend $H$ equivariantly to the $k$-skeleton such that (4) is satisfied.

To see that (4) $\Longrightarrow$ (2) observe that $H$ can be used to homotop every map $\varphi:(D^k,S^{k-1})\to (\tilde{X},N_i)$ to a map with image in any neighborhood of $\infty$. As the base point should not be moved during the homotopy, we have to modify $H$ on a `buffer zone' $N_{i-\lambda}-N_i$ so that points mapped to $N_i$ will not be changed. Nevertheless we can find $N_{i-\lambda}$ such that every $\varphi$ can be homotoped to a map $(D^k,S^{k-1})\to (N_{i-\lambda})$ which gives (1).
\end{proof}

%Condition (3) of Proposition \ref{equivdef} can be rephrased as saying that the direct system $\{\pi_l(\tilde{X},N)\}$ is ind-trivial for $l\leq k$.

\begin{remark}\label{arblarge}
Notice that in (4) we can choose $\varepsilon>0$ arbitrary large: as the homotopy is cellular, we can simply iterate it to increase the $\varepsilon$.
\end{remark}

Our next result shows that the Sigma invariants are in fact open subsets of $S(X)$, for the group theoretic version of this statement see \cite{binest,bieren}. For the proof we need a version of an Abel-Jacobi map. Let $q:\overline{X}\to X$ be the universal abelian cover of $X$ and let $r=b_1(X)$, the first Betti number of $X$. Then $H_1(X)$ acts on $\overline{X}$ by covering translations and on $\R^r=H_1(X)\otimes \R$ by translation.

There exists an equivariant map $A:\overline{X}\to \R^r$, canonical up to homotopy, called an \em Abel-Jacobi map\em, see \cite[Prop.1]{farscm}. For a different construction, note that we have a canonical epimorphism $\pi_1(X)\to \Z^r$ factoring through $H_1(X)$. Then $A$ is a lift of the resulting classifying map $X\to (S^1)^r$.

\begin{theorem}
Let $X$ be a finite connected CW-complex. For every $k\geq 0$ the set $\Sigma^k(X)$ is open and $\Sigma^n(X)=\Sigma^{\dim X}(X)$ for $n\geq \dim X$.
\end{theorem}

\begin{proof}
Let $h:\tilde{X}\to \R^r$ by the composition of the covering map $\bar{p}:\tilde{X}\to\overline{X}$ with the Abel-Jaobi map $A:\overline{X}\to \R^r$. Then, given $\xi:\pi_1(X)\to\R$ we get a height function by $h_\xi=l_\xi\circ h$ where $l_\xi:\pi_1(X)/[\pi_1(X),\pi_1(X)]\otimes \R\to \R$ is defined by $l_\xi([g]\otimes t)=\xi(g)\cdot t$. Now let $\tilde{H}$ be a homotopy as in Proposition \ref{equivdef} (4) for a $\xi$. Define $\tilde{K}:S(X)\times \tilde{X}^{(k)}\to\R$ by
\begin{eqnarray*}
 \tilde{K}(\xi',\tilde{x})&=&h_{\xi'}(H_1(\tilde{x}))-h_{\xi'}(\tilde{x}).
\end{eqnarray*}
By the choice of $H$ we get $\tilde{K}(\xi,\tilde{x})\geq \varepsilon$ for all $\tilde{x}\in\tilde{X}$ and some $\varepsilon>0$. It also induces a map $K:S(X)\times X^{(k)} \to \R$. From the compactness of $X^{(k)}$, we get
\begin{eqnarray*}
\tilde{K}(\xi',\tilde{x})&\geq &\frac{\varepsilon}{2}
\end{eqnarray*}
for all $\tilde{x}\in\tilde{X}^{(k)}$and all $\xi'$ in a neighborhood of $\xi$ in $S(X)$. Therefore Proposition \ref{equivdef} (4) is satisfied for all such $\xi'$.
\end{proof}

We get the following relation between $\Sigma^k(X)$ and the group-theoretic version $\Sigma^k(\pi_1(X))$.

\begin{proposition}
 Let $X$ be a finite connected CW-complex and $k\geq 0$. If $\tilde{X}$ is $k$-connected, then $\Sigma^k(X)=\Sigma^k(\pi_1(X))$ and $\Sigma^{k+1}(X)\subset \Sigma^{k+1}(\pi_1(X))$.
\end{proposition}

The inclusion can be proper, as the example $X=\Sp^1\vee \Sp^k$ with $k\geq 2$ shows.

\begin{proof}
 Note that we can build a $K(\pi_1(X),1)$ out of $X$ by attaching $n$-cells for $n\geq k+2$. Denote this Eilenberg-MacLane space by $Y$. If $\xi\in \Sigma^{k+1}(X)$, the cellular homotopy $\tilde{H}:\tilde{X}^{(k+1)}\times [0,1]\to \tilde{X}$ from Proposition \ref{equivdef} (4) induces a homotopy $\tilde{H}':\tilde{Y}^{(k+1)}\times [0,1] \to \tilde{Y}$, since $\tilde{Y}^{(k+1)}=\tilde{X}^{(k+1)}$ and $\tilde{X}\subset \tilde{Y}$. Therefore $\xi\in \Sigma^{k+1}(\pi_1(X))$.

If $\xi\in \Sigma^{k}(\pi_1(X))$, we get a cellular homotopy $\tilde{H}:\tilde{Y}^{(k)}\times [0,1]\to \tilde{Y}^{(k+1)}$ as in Proposition \ref{equivdef} (4), and since $\tilde{Y}^{(k+1)}=\tilde{X}^{(k+1)}$, this gives $\xi\in \Sigma^k(X)$.
\end{proof}

In analogy to the homological invariants $\Sigma^k(G,\Z)$ of Bieri and Renz \cite{bieren} we now want to define homological invariants $\Sigma^k(X,\Z)$. We will in fact introduce a more general definition for chain complexes which will also generalize the invariants of \cite{bieren}. We assume that all chain complexes satisfy $C_i=0$ for $i<0$.
\begin{definition}
Let $R$ be a ring, $n$ a non-negative integer and $C$ a chain complex over $R$. Then $C$ is of \em finite $n$-type\em, if there is a finitely generated projective chain complex $C'$ and a chain map $f:C'\to C$ with $f_i:H_i(C')\to H_i(C)$ an isomorphism for $i<n$ and an epimorphism for $i=n$. In this situation we call $f$ an $n$-equivalence.
\end{definition}

It is clear that this is equivalent to the existence of a free $R$-chain complex $D$ and a chain map $f:D\to C$ inducing an isomorphism on homology, and such that $D_i$ is finitely generated for $i\leq n$.

%Note that $S(X)$ only depends on $G=\pi_1(X)$, so if $G$ is a finitely generated group we can define $S(G)$ as above.

\begin{definition}
Let $C$ be a chain complex over $\Z G$ and $k\geq 0$. Then
\begin{eqnarray*}
\Sigma^k(C)&=&\{\xi\in S(G)\,|\, C \mbox{ is of finite }k\mbox{-type over }\Z G_\xi\}.
\end{eqnarray*}
\end{definition}
\begin{definition}
If $X$ is a finite connected CW-complex, we set
\begin{eqnarray*}
 \Sigma^k(X;\mathbb{Z})&=&\Sigma^k(C_\ast(\tilde{X})).
\end{eqnarray*}
\end{definition}

The invariants $\Sigma^n(G;A)$ of Bieri and Renz \cite{bieren} are given by $\Sigma^k(G;A)=\Sigma^k(P)$, where $P$ is a projective $\Z G$ resolution of the $\mathbb{Z}G$-module $A$.

\begin{remark}
Notice that $\Z G$ is a flat $\Z G_\xi$-module, as $\Z G$ is a direct limit of free $\Z G_\xi$-modules. As $\Z G\otimes_{\Z G_\xi}A\cong A$ for every $\Z G$-module $A$, we see that $\xi\in \Sigma^n(C)$ implies that $C$ is of finite $n$-type over $\Z G$. If the chain complex $C$ is free, we then get that $C$ is chain-homotopy equivalent to a free chain complex $D$ such that $D_i$ is finitely generated for $i\leq n$.
\end{remark}

To get an analogue of Proposition \ref{equivdef} we will define a chain complex version of a height function.

\begin{definition}
Let $C$ be a finitely generated free chain complex over $\Z G$ and $\xi:G\to \R$ a non-zero homomorphism.
A \em valuation on $C$ extending $\xi$ \em is a sequence of maps $v:C_k\to \R_\infty$ satisfying the following
\begin{enumerate}
\item $v(a+b)\geq \min\{v(a),v(b)\} $ for all $a,b\in C_k$.
\item $v(ga)=\xi(g)+v(a)$ for all $g\in G$, $a\in C_k$.
\item $v(-a)=v(a)$ for all $a\in C_k$.
\item $v(\partial a) \geq v(a)$ for all $a\in C_k$.
\item $v^{-1}(\{\infty\})=\{0\}$.
\end{enumerate}
\end{definition}
Here $\R_\infty$ denotes the reals together with an element $\infty$ with the obvious extension of addition and $\geq$.

To define a valuation on a free $\Z G$ complex $C$ which is finitely generated in every degree, let $X_i$ be a basis for $C_i$. We begin with setting $v(x)=0$ for all $x\in X_0$. The valuation can then be extended in the obvious way to $C_0$. Inductively we now define for $x\in C_i$ the valuation by $v(x)=0$ if $\partial x=0$, or $v(x)=v(\partial x)$, if $\partial x\not=0$. Again we can extend $v$ to $C_i$ which gives the existence of valuations on $C$.

\begin{proposition}\label{equivdefhom}
Let $X$ be a finite connected CW-complex, $\xi\in H^1(X;\R)$ be nonzero and $v:C_\ast(\tilde{X})\to\R_\infty$ a valuation extending $\xi$. The following are equivalent.
\begin{enumerate}
\item $\xi\in \Sigma^k(X,\Z)$.
\item There is a $\lambda\geq 0$ such that $j_\ast:H_l(\tilde{X},N_i)\to H_l(\tilde{X},N_{i-\lambda})$ is trivial for all $l\leq k$ and every $i\in\R$.
\item For every neighborhood $N$ of $\infty$ with respect to $\xi$, there is another neighborhood $N'\subset N$ such that $j_\ast:H_l(\tilde{X},N')\to H_l(\tilde{X},N)$ is trivial for all $l\leq k$.
\item Given $\varepsilon>0$ there exists a $\Z G$-chain map $A:C_\ast(\tilde{X})\to C_\ast(\tilde{X})$ chain homotopic to the identity with $v(A(x))\geq v(x)+\varepsilon$ for all non-zero $x\in C_i(\tilde{X})$ with $i\leq k$.
%
%\item Given a $\Z[\pi_1(X)]$-basis of the cellular chain complex $C_\ast(\tilde{X})$, there exists an $\varepsilon>0$ and a $\Z[\pi_1(X)]$-chain homotopy $H_l:C_l(\tilde{X})\to C_{l+1}(\tilde{X})$
%with $\partial_{l+1}H_l+H_{l-1}\partial_l=I_l-A_l$ and $\supp\,A_l\subset [\varepsilon,\infty)$ for all $l\leq k$.
\end{enumerate}
\end{proposition}

The equivalences of (2),(3) and (4) are similar to the proof of Proposition \ref{equivdef} and will be omitted. For the equivalence to (1) we refer to Appendix \ref{appendicitis} which treats a more general version.

\begin{corollary}
Let $X$ be a finite connected CW-complex. For every $k\geq 0$ the set $\Sigma^k(X,\Z)$ is open and $\Sigma^n(X,\Z)=\Sigma^{\dim X}(X,\Z)$ for $n\geq \dim X$.\qed
\end{corollary}

\begin{corollary}\label{homvshom}
Let $X$ be a finite connected CW-complex. Then
\begin{enumerate}
\item $\Sigma^1(X)=\Sigma^1(X,\Z)$.
\item $\Sigma^k(X)\subset \Sigma^k(X,\Z)$ for $k\geq 2$.
\end{enumerate}
\end{corollary}

\begin{proof}
It is easy to see that Condition (3) of Proposition \ref{equivdef} implies Condition (3) of Proposition \ref{equivdefhom} so that $\Sigma^k(X)\subset \Sigma^k(X,\Z)$ for $k\geq 1$. Also, for $k=1$ the chain homotopy of Proposition \ref{equivdefhom} (3) can easily be used to realize a homotopy as in Proposition \ref{equivdef} (3).
\end{proof}

It follows from the examples of Bestvina and Brady \cite{besbra} that in general $\Sigma^2(X) \not=\Sigma^2(X,\Z)$, see \cite{memewy}.

\section{Novikov rings and homology}\label{novring}
Let $G$ be a finitely generated group and $\xi\in S(G)$. We then let
\begin{eqnarray*}
 \widehat{\Z G}_\xi&=&\left\{\left.\sum_{g\in G}n_gg\,\right|\,\mbox{for all }t\in \R \,\,\,\#\{g\,|\,n_g\not=0\mbox{ and }\xi(g)>t\}<\infty\right\},
\end{eqnarray*}
a ring containing the group ring $\Z G$. If $\xi:G\to \R$ is injective, this is the \em Novikov ring \em first defined in \cite{noviko}. For general $\xi$ we call it the \em Novikov-Sikorav \em ring, first introduced in \cite{sikora}.

The relation to the Bieri-Neumann-Strebel-Renz invariants was immediately apparent once the work of Sikorav in \cite{sikora} became known. This relation also extends to our situation and is explained in Proposition \ref{sigmanov} below.

If $X$ is a finite connected CW-complex, $p:\overline{X}\to X$ a regular covering space with $\pi_1(\overline{X})\subset\ker\, \xi$, we get an induced homomorphism, also denoted by $\xi:G\to \R$, where $G=\pi_1(X)/\pi_1(\overline{X})$. We then obtain a finitely generated free $\widehat{\Z G}_\xi$-chain complex by setting
\begin{eqnarray*}
 C_\ast(X;\widehat{\Z G}_\xi)&=&\widehat{\Z G}_\xi\otimes_{\Z G}C_\ast(\overline{X})
\end{eqnarray*}
and we denote by
\begin{eqnarray*}
 H_\ast(X;\widehat{\Z G}_\xi)&=&H_\ast(C_\ast(X;\widehat{\Z G}_\xi))
\end{eqnarray*}
the resulting homology, called the \em Novikov-Sikorav homology \em(\em Novikov homology \em if $\xi$ is injective).

The case of the universal cover is the one directly related to $\Sigma^k(X)$, but the case of abelian covers plays an important role for the cup-length estimates of the Lusternik-Schnirelmann categories.

\begin{proposition}\label{sigmanov}
 Let $X$ be a finite connected CW-complex, $\xi\in H^1(X;\R)$ non-zero, $\tilde{X}$ the universal cover of $X$ and $G=\pi_1(X)$. Then the following are equivalent.
\begin{enumerate}
 \item $\xi\in \Sigma^k(X;\Z)$.
 \item $H_i(X;\widehat{\Z G}_{-\xi})=0$ for $i\leq k$.
\end{enumerate}
\end{proposition}

Note that for testing $\xi\in \Sigma^k(X;\Z)$ we have to use the Novikov-Sikorav ring with respect to $-\xi$. The reason for this is that we allow infinitely many non-zero coefficients in elements of $\widehat{\Z G}_\xi$ in the ``negative direction'', which is in line with the original definition of the Novikov ring \cite{N1}. But to adhere to the convention of \cite{bieren} one has to complete in the ``positive direction''. In order to stick to the convention used in \cite{fasc}, we have to introduce the minus-sign above.

\begin{proof}[Proof of Proposition \ref{sigmanov}]
 (1) $\Longrightarrow$ (2): Let $A:C_\ast(\tilde{X})\to C_\ast(\tilde{X})$ be the chain map chain homotopic to the identity, given by Proposition \ref{equivdefhom} (4). Then ${\rm id}-A:C_i(X;\widehat{\Z G}_{-\xi})\to C_i(X;\widehat{\Z G}_{-\xi})$ is an isomorphism with inverse ${\rm id}+A+A^2+\ldots$, which converges over $\widehat{\Z G}_{-\xi}$ by the valuation property in Proposition \ref{equivdefhom}(4). So the map on homology is both an isomorphism and zero, which means that the homology vanishes.

(2) $\Longrightarrow$ (1): As $C_\ast(X;\widehat{\Z G}_{-\xi})$ is free and bounded below, the vanishing of its homology groups up to degree $k$ guarantees the existence of a chain homotopy $\delta:C_\ast(X;\widehat{\Z G}_{-\xi})\to C_{\ast+1}(X;\widehat{\Z G}_{-\xi})$ with $\partial_{i+1}\delta_i+\delta_{i-1}\partial_i={\rm id}$ for $i\leq k$. ``Cutting off'' gives a chain homotopy $\bar{\delta}:C_\ast(\tilde{X})\to C_\ast(\tilde{X})$ with $\partial_{i+1}\bar{\delta}_i+\bar{\delta}_{i-1}\partial_i={\rm id}-A$. Then $A$ is a chain map homotopic to the identity, and by approximating $\delta$ sufficiently well with $\bar{\delta}$, condition (4) of Proposition \ref{equivdefhom} is satisfied.
\end{proof}

\section{Relations to the Lusternik-Schnirelmann category of closed 1-forms}\label{category}

We now want to describe a connection to the Lusternik-Schnirelmann theory of closed 1-forms, which has been introduced in a series of papers \cite{farber,farbe4,farkap} and for which more information, in particular on applications and calculations, can be found in \cite{farbook,fasc}. Indeed, there exist various different notions, and the one related closest to $\Sigma^k(X)$ is denoted by $\Cat(X,\xi)$. Here again $X$ is a finite CW-complex and $\xi\in H^1(X;\R)$. Let us recall the definition from \cite{fasc}.

\begin{definition}\label{bigcat}
 Let $X$ be a finite CW-complex and $\xi\in H^1(X;\R)$. Fix a closed 1-form $\omega$ representing $\xi$. Then $\Cat(X,\xi)$ is defined as the minimal integer $k$ such that there exists an open subset $U\subset X$ satisfying
\begin{enumerate}
 \item $\cat_X(X-U)\leq k$.
\item for some homotopy $h:U\times [0,\infty)\to X$ one has
\[h(x,0)=x\quad \mbox{ and }\quad \lim_{t\to\infty}\int_x^{h_t(x)}\omega=-\infty\]
for any point $x\in U$.
\item the limit in (2) is uniform in $x\in U$.
\end{enumerate}
The integral is taken along the path $\gamma:[0,t]\to X$ given by $\gamma(\tau)=h(x,\tau)$.
\end{definition}

Here $\cat_X(A)$ for $A\subset X$ is the minimal number $i$ such that there exist open sets $U_1,\ldots,U_i\subset X$ covering $A$, each of which is null-homotopic in $X$.
The invariant $\Cat(X,\xi)$ originally appeared in \cite{farbe4}. It is easy to see that it does not depend on positive multiples of $\xi$, therefore it is well defined for $\xi\in S(X)$. Note that it is also defined for $\xi=0$, in which case we recover the original Lusternik-Schnirelmann category of $X$, see \cite{fasc} for details.

The connection to $\Sigma^k(X)$ can now be described as follows. Because of our sign conventions, we obtain another minus-sign in front of a $\xi$.

\begin{theorem}\label{cattheo}
Let $X$ be a finite connected CW-complex and $\xi\in H^1(X;\R)$ be nonzero. Write $n=\dim X$. If $-\xi\in \Sigma^k(X)$ for some $k\leq n$. Then $\Cat(X,\xi)\leq n-k$.
\end{theorem}

\begin{proof}
Let $U$ be a small open neighborhood of $X^{(k)}$ in $X$ which deformation retracts to $X^{(k)}$. Let $H:X^{(k)}\times [0,1]\to X$ be a homotopy starting with inclusion which lifts to a homotopy as in Proposition \ref{equivdef} (3). Now define $h:X^{(k)}\times [0,\infty)\to X$ by $h(x,t)=H(H^m(x,1),t-m)$, where $m$ is an integer with $t\in[m,m+1]$. Notice that $H$ is cellular so $H^m(x,1)\in X^{(k)}$. It is easy to see that $h$ combined with the deformation retraction of $U$ to $X^{(k)}$ gives a homotopy as required in Definition \ref{bigcat}. It is well known that $\cat_X(X-X^{(k)})\leq \dim(X)-k$ so the result follows.
\end{proof}

There are other definitions for a Lusternik-Schnirelmann category of $\xi$, denoted by $\cat(X,\xi)$ and $\cat^1(X,\xi)$, see \cite{farber,farbook,farkap,fasc}, satisfying
\[
 \cat(X,\xi)\leq \cat^1(X,\xi)\leq \Cat(X,\xi).
\]
Therefore Theorem \ref{cattheo} provides an upper bound for these as well.

%\begin{question}
%Assume that $\xi\in \Sigma^k(X)$ where $k\geq 3$ (or $2$). Is it true that
%$\cat^1(X,\xi)=\Cat(X,\xi)$?
%\end{question}
%
%\begin{question}
%Find $\xi\in \Sigma^1(X)-\Sigma^2(X)$ such that $\Cat(X,\xi)>\cat^1(X,\xi)$.
%\end{question}

\section{A Hurewicz type result}\label{hurewi}

We can get a converse of Corollary \ref{homvshom} if we assume that $\xi\in \Sigma^2(\pi_1(X))$. This can be described in the following way.

\begin{definition}
Let $G$ be a finitely presented group and $\xi:G\to \R$ a nonzero homomorphism. Let $X$ be a finite connected CW-complex with $\pi_1(X)=G$. We say $\xi\in \Sigma^2(G)$, if there is a $\lambda\geq 0$ such that $j_\#:\pi_l(N_i)\to \pi_l(N_{i-\lambda})$ is trivial for $l\leq 1$ and every $i\in \R$.
\end{definition}

Note that we can think of $\xi\in H^1(X;\R)$ so neighborhoods of $\infty$ are defined as above.

\begin{theorem}\label{hurewicz}
Let $X$ be a finite connected CW-complex and $\xi\in H^1(X;\R)$ nonzero. Let $k\geq 2$. If $\xi\in \Sigma^2(\pi_1(X))\cap \Sigma^k(X,\Z)$, then $\xi\in \Sigma^k(X)$.
\end{theorem}

The theorem will follow from two lemmas.

\begin{lemma}\label{firstpart}
Let $X$ be a finite connected CW-complex and $\xi\in H^1(X;\R)$ nonzero. If $\xi\in \Sigma^2(\pi_1(X))$, then $\{\pi_2(\tilde{X},N)\}$ and $\{H_2(\tilde{X},N)\}$ are pro-isomorphic.
\end{lemma}

\begin{proof}
Notice that
\begin{eqnarray*}
\im(\pi_2(\tilde{X},N_i)\to \pi_2(\tilde{X},N_{i-\lambda}))&=&\im(\pi_2(\tilde{X})\to \pi_2(\tilde{X},N_{i-\lambda}))
\end{eqnarray*}
if every 1-sphere in $N_i$ bounds in $N_{i-\lambda}$, in particular the images become abelian. As $\xi\in\Sigma^2(\pi_1(X))$, we can define a homomorphism $\pi_2(\tilde{X},N_i)\to H_2(\tilde{X},N_{i-\lambda})$ as in a typical proof of the classical Hurewicz Theorem (after possibly increasing $\lambda$), see, for example, Spanier \cite{spanie}. The details will be left to the reader.
\end{proof}

\begin{lemma}\label{secondpart}
Let $X$ be a finite connected CW-complex and $\xi\in H^1(X;\R)$ nonzero. Let $k\geq 3$ and $\xi\in\Sigma^{k-1}(X)$. Then $\{\pi_k(\tilde{X},N)\}$ and $\{H_k(\tilde{X},N)\}$ are pro-isomorphic.
\end{lemma}

\begin{proof}
This is similar to the proof of Lemma \ref{firstpart}, but easier, as we can define the homomorphism directly.
\end{proof}

These two Lemmas combine to a proof of Theorem \ref{hurewicz}.

\begin{remark}
Theorem \ref{hurewicz} also follows from Latour's Theorem 5.10 and 5.3 \cite{latour}.%which are saying even more, namely that the Steenrod homotopy group and the Novikov-Sikorav homology group in the first nonvanishing dimension are isomorphic. The proofs are very long though.enrod homotopy group and the Novikov-Sikorav homology group in the first nonvanishing dimension are isomorphic. The proofs are very long though.
\end{remark}

\section{Functoriality properties}\label{functor}

\begin{proposition}\label{domer}
Let $X$ and $Y$ be finite CW-complexes, $f:X\to Y$ and $g:Y\to X$ maps with $fg\simeq {\rm id}_Y$. Then for all $k\geq 0$ we have
\begin{eqnarray*}
(f^\ast)^{-1}(\Sigma^k(X))&\subset &\Sigma^k(Y)\\
(f^\ast)^{-1}(\Sigma^k(X,\Z))&\subset &\Sigma^k(Y,\Z).
\end{eqnarray*}
\end{proposition}

\begin{proof}
We give a proof for the homotopy invariant $\Sigma^k(Y)$. Choose liftings $\tilde{f}:\tilde{X}\to\tilde{Y}$ and $\tilde{g}:\tilde{Y}\to \tilde{X}$ with $\tilde{f}\tilde{g}\simeq {\rm id}_{\tilde{Y}}$ equivariantly. Let $\xi\in S(Y)$ satisfy $f^\ast(\xi)\in \Sigma^k(X)$. Let $h_\xi:\tilde{Y}\to \R$ be a height function. Then $h_\xi\circ \tilde{f}:\tilde{X}\to\R$ is a height function for $f^\ast\xi$.

By cocompactness, there exists a $C\geq 0$ such that
\begin{eqnarray*}
|h_\xi(\tilde{f}\tilde{g}(\tilde{y}))-h_\xi(\tilde{y})|&\leq &C
\end{eqnarray*}
for every $\tilde{y}\in \tilde{Y}$.

By assumption, Proposition \ref{equivdef} and Remark \ref{arblarge}, there is an equivariant homotopy $\tilde{H}:\tilde{X}^{(k)}\times [0,1]\to \tilde{X}$ starting with inclusion, such that
\begin{eqnarray*}
h_\xi\tilde{f}(H(\tilde{x},1))-h_\xi\tilde{f}(\tilde{x})&\geq &C+1
\end{eqnarray*}
for all $\tilde{x}\in\tilde{X}$. Therefore
\begin{eqnarray*}
h_\xi\tilde{f}(H(\tilde{g}(\tilde{y}),1))-h_\xi(\tilde{y})&\geq & 1.
\end{eqnarray*}
Now $\tilde{f}H(\tilde{g},\cdot)$ can be combined with the homotopy $\tilde{f}\tilde{g}\simeq {\rm id}_{\tilde{Y}}$ to show that $\xi\in\Sigma^k(Y)$.
\end{proof}

\begin{corollary}
Let $X$ and $Y$ be finite connected CW-complexes and $h:X\to Y$ a homotopy equivalence. Then $h^\ast(\Sigma^k(Y))= \Sigma^k(X)$ and $h^\ast(\Sigma^k(Y,\Z))= \Sigma^k(X,\Z)$ for all $k\geq 0$.
\end{corollary}

\begin{proposition}\label{mconnect}
Let $X$ and $Y$ be finite CW-complexes and $f:X\to Y$ $m$-connected with $m\geq 1$. Then
\begin{eqnarray*}
(f^\ast)^{-1}(\Sigma^k(X))&\subset &\Sigma^k(Y)\\
(f^\ast)^{-1}(\Sigma^k(X,\Z))&\subset &\Sigma^k(Y,\Z)
\end{eqnarray*}
for all $k\leq m$.
\end{proposition}

\begin{proof}
Add cells of dimension $\geq m+1$ to $X$ to get a (possibly infinite) CW-complex $X'$ containing $X$ such that $f$ extends to a homotopy equivalence $f':X'\to Y$. Let $g:Y\to X'$ be a cellular homotopy inverse. Then $g(Y^{(m)})\subset X$. If $f^\ast(\xi)\in\Sigma^k(X)$, the homotopy $H:\tilde{X}^{(k)}\times [0,1]\to \tilde{X}$ can now be used as in the proof of Proposition \ref{domer} to show that $\xi\in\Sigma^k(Y)$ for $k\leq m$.
\end{proof}

\begin{example}\label{mappingtor}
Let $X$ be a finite connected CW-complex and $f:X\to X$ a map. The mapping torus $M_f$ is the quotient space $M_f=X\times [0,1]/\,\sim$, where $(x,0)\sim (f(x),1)$. There is a natural map $g:M_f\to \Sp^1$ given by $g([x,t])=\exp 2\pi it$. Let
\[\xi=[g]\in [M_f,\Sp^1]=H^1(M_f;\Z)\subset H^1(M_f;\R).\]
The homotopy $h:M_f\times [0,1]\to M_f$ given by
\begin{eqnarray*}
H([x,t],s)&=&\left\{\begin{array}{rc}[x,t-s]&t\geq s\\ \,[f(x),1+t-s]&t\leq s\end{array}
\right.
\end{eqnarray*}
shows that $-\xi\in \Sigma^k(M_f)$ for all $k\geq 0$.
\end{example}

\begin{proposition}\label{firstcase}
Let $X$ be a finite connected CW-complex and $\xi\in H^1(X;\Z)$ be nonzero. Let $q:\bar{X}\to X$ be the infinite cyclic covering space corresponding to $\ker\,\xi$. Assume that $\bar{X}$ is homotopy equivalent to a CW-complex $Y$ with finite $k$-skeleton. Then $\pm \xi\in \Sigma^k(X)$.
\end{proposition}

\begin{proof}
Let $h:\bar{X}\to\R$ be induced by a height function $h_\xi:\tilde{X}\to\R$ and $\zeta:\bar{X}\to\bar{X}$ the generating covering transformation with $h\zeta(x)>h(x)$ for all $x\in\bar{X}$. Let $a:Y\to \bar{X}$ and $b:\bar{X}\to Y$ be mutually inverse homotopy equivalences. The there is a homotopy equivalence $g:M_{b\zeta a}\to X$ given by
\[M_{b\zeta a}\simeq M_{\zeta ab}\simeq M_\zeta \simeq X\]
where the last homotopy equivalence is given by $[\bar{x},t]\to q(\bar{x})$ for $\bar{x}\in\bar{X}$.

We can assume that $b\zeta a:Y\to Y$ sends the $k$-skeleton to the $k$-skeleton. Let $\varphi:Y^{(k)}\to Y^{(k)}$ be the restriction of $b\zeta a$ to $Y^{(k)}$. The induced map $M_\varphi\to M_{b\zeta a}$ is $k$-connected and so there is a $k$-connected map $f:M_\varphi\to X$. By Example \ref{mappingtor} $-f^\ast(\xi)\in \Sigma^k(M_\varphi)$. It follows from Proposition \ref{mconnect} that $-\xi\in\Sigma^k(X)$. To get $\xi\in \Sigma^k(X)$ as well, replace $\zeta$ by $\zeta^{-1}$.
\end{proof}

In the next section we will show that the converse of Proposition \ref{firstcase} also holds.

\section{Domination results for covering spaces}\label{mainth}

Let $X$ be a finite connected CW-complex and $q:\overline{X}\to X$ a regular covering  space with $\pi_1(X)/\pi_1(\overline{X})$ abelian. Then we define
\begin{eqnarray*}
S(X,\overline{X})&=&\{\xi\in S(X)\,|\, q^\ast\xi=0\}.
\end{eqnarray*}
In particular, if $\overline{X}$ is the universal abelian covering of $X$, then $S(X,\overline{X})=S(X)$. More generally, $S(X,\overline{X})$ is a sphere of dimension $d-1$, where $d$ is the rank of the finitely generated abelian group $\pi_1(X)/\pi_1(\overline{X})$.

\begin{theorem}\label{dominationX}
Let $X$ be a finite connected CW-complex and $q:\overline{X}\to X$ a regular covering  space with $\pi_1(X)/\pi_1(\overline{X})$ abelian. Let $k\geq 1$. Then the following are equivalent.
\begin{enumerate}
\item $\overline{X}$ is homotopy equivalent to a CW-complex with finite $k$-skeleton.
\item $S(X,\overline{X})\subset \Sigma^k(X)$.
\end{enumerate}
Furthermore, if $S(X,\overline{X})\subset\Sigma^{\dim X}(X)$, then $\overline{X}$ is finitely dominated.
\end{theorem}

We derive Theorem \ref{dominationX} from a more general version for chain complexes which is a generalization of \cite[Thm.B]{bieren}. To get an alternative proof for (1) $\Longrightarrow$ (2) one can use the techniques of \cite[Thm.3.2]{schfin}.

If $N$ is a normal subgroup of $G$ with $G/N$ abelian, we write
\begin{eqnarray*}
S(G,N)&=&\{\xi\in S(G)\,|\,N\leq\ker\,\xi\}
\end{eqnarray*}

\begin{theorem}\label{dominator}
Let $C$ be a free $\Z G$-chain complex which is finitely generated in every degree $i\leq n$, and $N$ a normal sugroup of $G$ such that $G/N$ is abelian. Then $C$ is of finite $n$-type over $\Z N$ if and only if $S(G,N)\subset \Sigma^n(C)$.
\end{theorem}

Theorem \ref{dominationX} follows from Theorem \ref{dominator} by the work of Wall \cite{wall,wall2}.

\begin{proof}
Assume that $C$ is of finite $n$-type over $\Z N$. Let $\xi\in S(G,N)$ and denote $Q=G/N$. Then $\xi$ induces a homomorphism, also denoted $\xi$, $\xi:Q\to\R$. As $\Z G$ is free over $\Z N$, there is a chain homotopy equivalence $f:P\to C$ over $\Z N$ with $P_j$ finitely generated free for all $j\leq n$. Then $f:\Z G_\xi\otimes_{\Z N}P\to \Z G_\xi\otimes_{\Z N} C$ shows that $\Z G_\xi\otimes_{\Z N} C$ is of finite $n$-type over $\Z G_\xi$. Also $\Z G_\xi\otimes_{\Z N} C\cong \Z Q_\xi\otimes C$ where the chain complex on the right has $\Z G_\xi$ acting diagonally. The isomorphism is given by $g\otimes c\mapsto \pi(g)\otimes gc$, where $\pi:G\to Q$ is projection.
By \cite[Lm.5.2]{bieren} there is a free resolution $E_\ast\to \Z$ over $\Z Q_\xi$ which is finitely generated in every degree. Therefore each $E_p\otimes C$ is of finite $n$-type over $\Z G_\xi$. Let $f:P_{p\,q}\to E_p\otimes C_q$ be the corresponding chain map, notice that $P_{p\,q}$ is just a positive power of $\Z G_\xi\otimes_{\Z N}P_q$ depending on the rank of $E_p$. As $C$ is free over $\Z N$, we get that $\Z Q_\xi\otimes C\cong \Z G_\xi\otimes_{\Z N}C$ is free over $\Z G_\xi$, and we can assume that $f$ is a chain homotopy equivalence with inverse $g_q:E_p\otimes C_q\to P_{p\,q}$. Denote by $L:E_p\otimes C_q\to E_p\otimes C_{q+1}$ the chain homotopy $L:fg\simeq 1$.

For $k\geq 0$ define
\[F^k:P_{p\,q}\to E_{p-k}\otimes C_{q+k} \mbox{ by }F^k=(Ld)^kf,\]
where $d:E_p\to E_{p-1}$ is the boundary of the resolution. Also define
\[K^i:P_{p\,q}\to P_{p-i\,q-1+i} \mbox{ by }K^i=gdF^{i-1}\]
for $i\geq 1$. We also set $K^0=\partial:P_{p\,q}\to P_{p-1\,q}$.

\begin{lemma}\label{chainf}
Let $\partial$ denote the boundaries $\partial:P_{p\,q}\to P_{p\,q-1}$ and $\partial:E_p\otimes C_q\to E_p\otimes C_{q-1}$. Then $\partial F^0=F^0\partial$ and for $m\geq 1$ we have
\begin{eqnarray*}%\label{chainhomF}
\partial F^m+(-1)^{m+1} F^m\partial&=&\sum_{k=0}^{m-1}(-1)^k F^kK^{m-k} - dF^{m-1}.
\end{eqnarray*}
\end{lemma}

\begin{proof}
The proof by induction is straightforward.
\end{proof}

\begin{lemma}\label{sumzero}
For $m\geq 0$ we have
\begin{eqnarray*}
\sum_{s=0}^m (-1)^s K^{m-s}K^s&=&0.
\end{eqnarray*}
\end{lemma}

\begin{proof}
For $m=0$ this means $\partial\partial=0$, so assume the statement holds for $m\geq 0$. Then
\begin{eqnarray*}
\sum_{s=0}^{m+1} (-1)^s K^{m-s}K^s&=&gdF^m\partial +(-1)^{m+1} \partial gd F^m +\sum_{s=1}^m (-1)^s K^{m-s}K^s\\
&=&(-1)^{m+1}\left(gd\sum_{k=0}^{m-1}(-1)^k F^kK^{m-k} - gddF^{m-1}\right) +\\
& & + \sum_{s=1}^m (-1)^s gdF^{m-s-1}K^s\\
&=&0
\end{eqnarray*}
by Lemma \ref{chainf} and since $dd=0$.
\end{proof}

Now define a chain complex $TP$ by $TP_k=\bigoplus\limits_{p+q=k}P_{p\,q}$ and $\delta:TP_k\to TP_{k-1}$ by
\begin{eqnarray*}
\delta&=&\sum_{s=0}^\infty (-1)^pK^s
\end{eqnarray*}
where $(-1)^p$ refers to $P_{p\,k-p}$. By Lemma \ref{sumzero} we get that $\delta\delta=0$. Also define $TE_k=\bigoplus\limits_{p+q=k}E_p\otimes C_q$ with $\delta=(-1)^p(\partial+d)$. We get a chain map $F:TP\to TE$ by setting
\begin{eqnarray*}
F&=&\sum_{k=0}^\infty (-1)^k F^k
\end{eqnarray*}
That $F$ is indeed a chain map follows from Lemmata \ref{chainf} and \ref{sumzero}.

Using the filtrations $(TP^{(m)})_k=\bigoplus\limits_{p+q=k, p\leq m}P_{p\,q}$ and \\ $(TE^{(m)})_k=\bigoplus\limits_{p+q=k, p\leq m}E_p\otimes C_q$ we see that $F$ induces a chain homotopy equivalence $\bar{F}:TP^{(m)}/TP^{(m-1)}\to TE^{(m)}/TE^{(m-1)}$ and by a spectral sequence argument $F$ is a chain homotopy equivalence. Another spectral sequence argument gives a homology isomorphism from $TE$ to $C$. As $TP_k$ is finitely generated free over $\Z G_\xi$, we now get that $C$ is of finite $n$-type over $\Z G_\xi$.

Now assume that $S(G,N)\subset \Sigma^n(C)$.

Let $X_i$ be a $\Z G$-basis of $C_i$ for $i\leq n$, a finite set by assumption. Given $c\in C_i$, we can therefore write
\begin{eqnarray*}
	c&=&\sum_{j=1}^{m_i}n^i_jx^i_j
\end{eqnarray*}
with $n^i_j\in\Z G$ and
\begin{eqnarray*}
	\partial c&=&\sum_{j=1}^{m_{i-1}}n^{i-1}_jx_j^{i-1}
\end{eqnarray*}
here the $x^i_j$ denote the elements of $X_i$. For $y\in\Z G$ we denote ${\rm supp}\,y$ as the elements of $G$ with nonzero coefficient and for $c\in C_i$ as above we let
\begin{eqnarray*}
	{\rm supp}\,c&=&\bigcup_{j=1}^{m_i}{\rm supp}\,n_j^i\cup\bigcup_{j=1}^{m_{i-1}}{\rm supp}\,n_j^{i-1}
\end{eqnarray*}
In particular, we get ${\rm supp}\,\partial c\subset {\rm supp}\,c$. The support depends on the chosen basis, but we fix the basis once and for all.

We also denote $\pi:G\to G/N=Q$. By choosing an inner product on $Q_\mathbb{R}=\mathbb{R}\otimes Q$ we get a norm $\|\cdot\|$ on $Q_\mathbb{R}$ and we can think of $S(G,N)\subset Q_\mathbb{R}$ as the unit sphere in this normed vector space. We extend the norm to $C$ by setting
\begin{eqnarray*}
	\|c\|&=&\left\{\begin{array}{cl}\max\{\|\pi(g)\|\,|\,g\in{\rm supp}\,c\}&c\not=0\\
	0&c=0\end{array}\right.
\end{eqnarray*}
Notice that we set $\|0\|=0$ despite the fact that $\|0\|=-\infty$ in \cite{bieren}. For $\xi\in S(G,N)\subset Q_\mathbb{R}$ we also obtain a valuation $v_\xi$ by setting
\begin{eqnarray*}
	v_\xi(c)&=&\left\{\begin{array}{cl}\min\{\langle\pi(g),\xi\rangle\,|\,g\in {\rm supp}\,c\}&c\not=0\\
	\infty&c=0\end{array}\right.
\end{eqnarray*}
Let us also set for $a,b\in C$
\begin{eqnarray*}
	\diam(a,b)&=&\max\{\|\pi(g)-\pi(h)\|\,|\,g\in{\rm supp}\,a,h\in{\rm supp}\,b\}
\end{eqnarray*}
with the convention that $\diam(a,b)=0$ if $a$ or $b$ is zero. Finally, for $r>0$ and $c\in C$ we let
\begin{eqnarray*}
	B_r(c)&=&\{d\in C\,|\,\diam(c,d)\leq r\}.
\end{eqnarray*}

Given $r>0$, as $S(G,N)\subset \Sigma^n(C)$ we can find for every $\xi\in S(G,N)$ a chain map $\varphi^\xi:C\to C$ and a chain homotopy $H^\xi:1\simeq \varphi^\xi$ such that $v_\xi(\varphi^\xi c)-v_\xi(c)\geq 2r$ for all $c\in C$. Furthermore, there is an open neighborhood $U_\xi$ of $\xi$ in $S(G,N)$ such that
\begin{eqnarray*}
	v_\eta(\varphi^\xi c)-v_\eta(c)&\geq &r
\end{eqnarray*}
for all $\eta\in U_\xi$ and $c\in C_l$, $l\leq n$.

As $S(G,N)$ is compact, finitely many of the $U_\xi$ suffice to cover $S(G,N)$, so let $(U_i,\varphi^i,H^i)$ for $i=1,\ldots ,k$ be triples where $U_i$ cover $S(G,N)$, $\varphi^i:C\to C$ chain map with
\begin{eqnarray*}
	v_\eta(\varphi^i c)-v_\eta(c)&\geq &r
\end{eqnarray*}
for all $\eta \in U_i$ and all $c\in C_l$, $l\leq n$, and $H^i:1\simeq \varphi^i$ a chain homotopy.

As we deal only with finitely many chain homotopies, there is a $M\geq 0$ such that
\begin{eqnarray*}
	v_\eta(c)-v_\eta(H^i c)&\leq &M
\end{eqnarray*}
for all $\eta\in U_i$ and all $c\in C_l$, $l\leq n$.
Furthermore, by replacing the chain homotopy $H^i$ by $H^i-\varphi^iH^i:1\simeq (\varphi^i)^2$ we can increase $r>0$ without increasing $M>0$. Therefore we can assume that
\begin{eqnarray}\label{bigr}
	r&>&3Mn
\end{eqnarray}

Finally, there exists an $L>0$ with
\begin{eqnarray*}
	\diam(x,x)&\leq&L\\
	\diam(x,\varphi^ix)&\leq&L\\
	\diam(x,H^ix)&\leq&L
\end{eqnarray*}
for all $i=1,\ldots,k$ and all $x\in X_l$, $l\leq n$, as there are only finitely many such conditions.

Notice that $C$ is a free $\Z N$ chain complex, and a basis is given by $TX_i=\{tx\,|\,t\in T, x\in X_i\}$, where $T\subset G$ is a subset such that $\pi|T$ induces a bijection from $T$ to $Q$.

\begin{lemma}\label{pushin}
If $S(G,N)\subset \Sigma^n(C)$, there exist constants $r>0$, $M>0$, $A>0$, a $\Z N$-chain map $\psi:C\to C$ and a $\Z N$-chain homotopy $K:1\simeq \psi$ such that for $m\leq n$ we have $\psi_m(z)=z$ if $\|z\|\leq A$, and for $tx\in TX_m$ with $\|tx\|>A$ we get
\begin{eqnarray*}
	\|\psi_m(tx)\|&\leq &\|tx\|-r.
\end{eqnarray*}
Furthermore $\|K_m(z)\|\leq \|z\|+M$ for all $z\in C_m$.
\end{lemma}
\begin{proof}
In order to define $\psi$ and $K$, we define them on $TX_m$. Let $r>0$ and $M>0$ be as above Lemma \ref{pushin}.

\begin{lemma}\label{pusher}
Let $tx\in TX_m$ with $\|tx\|\geq \max\{\frac{3}{4}r+\frac{L^2}{r}, \frac{L^2}{M}\}$. Let $i$ be such that $\xi_t/\|\xi_t\|\in U_i$, where $\xi_t(g)=\langle \pi(g),\pi(t)\rangle$ for $g\in G$. Then
\begin{eqnarray*}
	\|\varphi^i(tx)\|& \leq & \|tx\|-\frac{1}{2}r\\
	\|H^i(tx)\|&\leq &\|tx\|+\frac{3}{2}M.
\end{eqnarray*}
\end{lemma}
\begin{proof}
Let $g\in{\rm supp}\,tx$ satisfy $\|\pi(g)\|=\|tx\|$. As $\diam(tx,\varphi^i tx)\leq L$, we get for $h\in {\rm supp}\,\varphi^i(tx)$
\begin{eqnarray*}
	\|\pi(h)\|^2&\leq & (\|\pi(g)\|-r)^2+L^2\\
	&=&\|\pi(g)\|^2-2r\|\pi(g)\|+r^2+L^2\\
	&\leq& \|\pi(g)\|^2-r\|\pi(g)\|-\frac{3}{4}r^2-L^2+r^2+L^2\\
	&=&(\|\pi(g)\|-\frac{r}{2})^2
\end{eqnarray*}
So $\|\varphi^i(tx)\|\leq \|tx\|-\frac{r}{2}$. Similarly, using $\|\pi(g)\|\geq \frac{L^2}{M}$, we get $H^i(tx)\|\leq \|tx\|+\frac{3}{2}M$.
\end{proof}
If $\|tx\|>L$, then $\pi({\rm supp}\,tx)\subset Q_\mathbb{R}-\{0\}$. Denote $p:Q_\mathbb{R}-\{0\}\to S(G,N)$ the standard retraction. Using the Lebesgue number of the covering $U_1\cup\ldots\cup U_k$, we can find a constant $A'>0$ with $\|tx\|>A'$ implying
\begin{eqnarray}\label{bigproj}
	p(\pi({\rm supp}\, B_{L'}(tx)))&\subset & U_i
\end{eqnarray}
for some $i$, where $L'=\max\{\frac{3}{4}r+\frac{L^2}{r}, \frac{L^2}{M}\}+(2n+1)L$. Let $A=A'+L'$.

Define $K_0:C_0\to C_1$ by $K_0(tx)=0$ if $\|tx\|\leq A$, and $K_0(tx)=H^i_0(tx)$ if $\|tx\|>A$, where $i$ is the smallest number with (\ref{bigproj}) satisfied. Also define $\psi_0:C_0\to C_0$ by
\begin{eqnarray*}
	\psi_0(tx)&=&\left\{\begin{array}{ccc}tx&\mbox{if}&\|tx\|\leq A\\ \varphi^i(tx) & \mbox{if}&\|tx\|>A\end{array}\right.
\end{eqnarray*}
where $i$ is again the smallest number such that (\ref{bigproj}) holds. It follows that $\partial K_0=1-\psi_0$. By Lemma \ref{pusher} we get
\begin{eqnarray*}
	\|\psi_0(tx)\|&\leq&\|tx\|-\frac{r}{2}\\
	\|K_0(tx)\|&\leq &\|tx\|+\frac{3}{2}M
\end{eqnarray*}
for $\|tx\|>A$. Also
\begin{eqnarray*}
	\diam(\psi_0(tx),tx)&\leq & L\\
	\diam(K_0(tx),tx)&\leq & L
\end{eqnarray*}
for all basis elements $tx$.

Assume now by induction that we have $K_j:C_j\to C_{j+1}$, $\psi_j:C_j\to C_j$ with 
\begin{eqnarray*}
\partial K_j+K_{j-1}\partial=1-\psi_j
\end{eqnarray*}
and
\begin{eqnarray}\label{psiline}
	\|\psi_j(tx)\|&\leq&\|tx\|-\frac{r}{2}+j\frac{3}{2}M\\
	\|K_j(tx)\|&\leq &\|tx\|+\frac{3}{2}M(j+1)\label{kline}
\end{eqnarray}
for $\|tx\|>A$ and
\begin{eqnarray}\label{diama}
	\diam(\psi_j(tx),tx)&\leq & (2j+1)L\\
	\diam(K_j(tx),tx)&\leq & (2j+1)L\label{diamb}
\end{eqnarray}
for all basis elements $tx$, for all $j\leq m-1$ with $m\leq n$.

Then define $K_m:C_m\to C_{m+1}$ by $K_m(tx)=0$ if $\|tx\|\leq A$ and
\begin{eqnarray*}
	K_m(tx)&=&H^i_m(tx)-H^i_m K_{m-1}\partial(tx)
\end{eqnarray*}
for $\|tx\|>A$, where $i$ satisfies (\ref{bigproj}). It is straightforward to check that
\begin{eqnarray*}
	(\partial K_m+K_{m-1}\partial)(tx)&=&tx-\varphi^i_m(tx)+\varphi^i_m(K_{m-1}\partial tx)-H^i_{m-1}\psi_{m-1}\partial tx
\end{eqnarray*}
for $\|tx\|>A$. So define $\psi_m:C_m\to C_m$ by $\psi_m(tx)=tx$ for $\|tx\|\leq A$ and
\begin{eqnarray*}
	\psi_m(tx)&=&\varphi^i_m(tx)-\varphi_m^i(K_{m-1}\partial tx)+H^i_{m-1}\psi_{m-1}\partial tx
\end{eqnarray*}
if $\|tx\|>A$. It follows that $\partial K_m+K_{m-1}\partial=1-\psi_m$. As
\[ \|\varphi^i_m(tx)-\varphi_m^i(K_{m-1}\partial tx)+H^i_{m-1}\psi_{m-1}\partial tx\|\,\,\,\leq\hspace{2cm}\]
\[	\max\{\|\varphi^i_m(tx)\|, \|\varphi_m^i(K_{m-1}\partial tx)\|,\|H^i_{m-1}\psi_{m-1}\partial tx\|\}
\]
it is easy to see that (\ref{psiline}) also holds for $j=m$. Here we use that \\ $\diam(K_{m-1}\partial tx,tx)\leq (2m-1)L+L$, so that Lemma \ref{pusher} still applies to all basis elements occuring in $K_{m-1}\partial tx$ by the choice of $A$.

Similarly, (\ref{kline}), (\ref{diama}) and (\ref{diamb}) hold for $j=m$. For $m>n$ we set $K_m=0$ and define $\psi_m$ such that
\begin{eqnarray*}
\partial K_m+K_{m-1}\partial=1-\psi_m.
\end{eqnarray*}
As the identity is a chain map, we get that $\psi$ is a chain map. Replacing $r$ by $\frac{r}{2}-\frac{3}{2}nM$ and $M$ by $\frac{3}{2}M(n+1)$ we get Lemma \ref{pushin}. Note that the new $r>0$ by (\ref{bigr}).
\end{proof}

\begin{lemma}\label{almost}
If $S(G,N)\subset \Sigma^n(C)$, then there exists a free $\Z N$ chain complex $D$ with $D_i$ finitely generated for $i\leq n$, and $\Z N$-chain maps $a:D\to C$, $b:C\to D$ with $ab$ chain homotopic to the identity on $C$. Also, $a:D_i\to C_i$ can be assumed to be inclusion for $i\leq n$, and $D_i=C_i$ for $i>n$.
\end{lemma}

\begin{proof}
Let $r, M,A>0$, $\psi$ and $K$ as provided by Lemma \ref{pushin}. If $l$ is a positive integer, $B\geq 0$ and $z\in C_m$ with $m\leq n$, we have
\begin{equation}
	\|\psi^l(z)\|\leq A+B \hspace{0.8cm}\mbox{if}\hspace{0.8cm}\|z\|\leq A+l\cdot r+B
	\label{reducer}
\end{equation}
We define a $\Z N$-chain homotopy $\Phi:C\to C_{+1}$ as follows. For $s\leq n$, set
\begin{eqnarray*}
	\Phi_s(tx)&=&\sum_{j=0}^l K_s\psi^jtx-K_s\psi^j\Phi_{s-1}\partial tx
\end{eqnarray*}
where $l$ is an integer such that $\|tx\|\in (A+l\cdot r,A+(l+1)r]$, and for $s>n$ we simply set $\Phi_s=0$. We get a $\Z N$-chain map $\zeta:C\to C$ by setting $\zeta=1-\partial\Phi-\Phi\partial$, in particular, $\zeta$ is chain homotopy equivalent to the identity.

Using induction, we see that
\begin{eqnarray*}
	\|\Phi_s(tx)\|&\leq&\|tx\|+(s+1)M
\end{eqnarray*}
for $s\leq n-1$. It follows that
\begin{eqnarray*}
	\|\Phi_s(z)\|&\leq&\|z\|+(s+1)M
\end{eqnarray*}
for all $z\in C_s$.

We claim that for $s\leq n$ we get
\begin{eqnarray}
	\label{starstar}
	\|\zeta(z)\|&\leq&A+s\cdot M
\end{eqnarray}
for all $z\in C_s$.

This holds for $s=0$ by (\ref{reducer}), as $\zeta(tx)=\psi^{l+1}(tx)$ if $\|tx\|\in(A+l\cdot r,A+(l+1)r]$.

Now notice that
\begin{eqnarray*}
	(\partial\Phi_s+\Phi_{s-1}\partial)(tx)&=&\sum_{j=0}^l (K_s\psi^jtx-K_s\psi^j\Phi_{s-1}\partial tx)+\Phi_{s-1}\partial(tx)\\
	&=&\sum_{j=0}^l((\partial K_s\psi^j+K_{s-1}\partial \psi^j)(tx)\\
	& &-\sum_{j=0}^l(\partial K_s\psi^j\Phi_{s-1}+K_{s-1}\partial\psi^j\Phi_{s-1})(\partial tx)\\
	& &+\sum_{j=0}^l (K_{s-1}\psi^j\partial\Phi_{s-1}\partial tx - K_{s-1}\partial\psi^jtx)+\Phi_{s-1}\partial tx\\
	&=&tx-\psi^{l+1}tx-\Phi_{s-1}\partial tx+\psi^{l+1}\Phi_{s-1}\partial tx\\
	& &-\sum_{j=0}^l K_{s-1}\partial\psi^jtx+\Phi_{s-1}\partial tx\\
	& &+\sum_{j=0}^l (K_{s-1}\psi^j\partial tx- K_{s-1}\psi^j\zeta\partial tx)\\
	&=&tx-\psi^{l+1}tx+\psi^{l+1}\Phi_{s-1}\partial tx -\sum_{j=0}^lK_{s-1}\psi^j\zeta\partial tx
\end{eqnarray*}
It follows that
\begin{eqnarray*}
	\zeta_s(tx)&=&\psi^{l+1}tx-\psi^{l+1}\Phi_{s-1}\partial tx+\sum_{j=0}^lK_{s-1}\psi^j\zeta_{s-1}\partial tx
\end{eqnarray*}
Using induction, we see again by (\ref{reducer}), that (\ref{starstar}) holds.

Define a chain complex $D$ by
\begin{eqnarray*}
	D_s&=&\{z\in C_s\,|\,\|z\|\leq A+nM\}
\end{eqnarray*}
for $s\leq n$. For $s>n$ we let $D_s=C_s$. The boundary map $\partial_D:D_{n+1}\to D_{n}$ is given by $\partial_D=\zeta\circ \partial_C$. We can define a chain map $b:C\to D$ by using $\zeta$ for $s\leq n$ and the identity for $s>n$ and a chain map $a:D\to C$ by using inclusion for $s\leq n$ and $\zeta$ for $s>n$. Then $ab=\zeta:C\to C$ is chain homotopic to the identity. As only finitely many $q\in Q$ satisfy $\|q\|\leq A+nM$, we get that $D_s$ is finitely generated free over $\Z N$, compare \cite[5.4]{bieren}. This finishes the proof of Lemma \ref{almost}.
\end{proof}
To finish the proof of Theorem \ref{dominator}, note that by Lemma \ref{almost} the chain complex $C$ is dominated over $\Z N$ by a chain complex $D$ with $D_i$ finitely generated free for $i\leq n$. By a standard construction, compare Ranicki \cite[\S 3]{ranims} or Wall \cite{wall2}, there exists a $\Z N$ chain complex $E$ chain homotopy equivalent to $C$ with $E_i$ finitely generated free for $i\leq n$.
\end{proof}

\begin{corollary}
Let $C$ be a finitely generated free $\Z G$-chain complex with $C_i=0$ for $i>n$. Then $C$ is $\Z N$-chain homotopy equivalent to a finitely generated projective $\Z N$-chain complex $D$ with $D_i=0$ for $i>n$ if and only if $S(G,N)\subset \Sigma^n(C)$.
\end{corollary}

\begin{proof}
If $C$ is $\Z N$-chain homotopy equivalent to a finitely generated projective $\Z N$-complex, it is, by definition, of finite $n$-type over $\Z N$, hence $S(G,N)\subset \Sigma^n(C)$ by Theorem \ref{dominator}.

If $S(G,N)\subset \Sigma^n(C)$, then by Lemma \ref{almost} there is a chain complex $D'$, finitely generated free over $\Z N$ with $D_i'=0$ for $i>n$ which dominates $C$. By \cite[Prop.3.1]{ranims}, the required chain complex $D$ exists.
\end{proof}

\section{Movability of homology classes}\label{movhomnotion}
Recall from Proposition \ref{sigmanov} that the vanishing and non-vanishing of Novikov-Sikorav homology groups in the universal cover case determine $\Sigma^k(X,\Z)$. We now want to take a closer look at the homology of other coverings, in particular coverings with abelian covering transformation group.

In this section, $R$ is a ring, although we have mainly the cases $R=\Z$ and $R=\kk$ a field in mind.

Let $p:\overline{X}\to X$ be a regular cover, where $X$ is again a finite connected CW-complex. Denote $G=\pi_1(X)/\pi_1(\overline{X})$. Recall that $S(X,\overline{X})$ consists of those $\xi\in S(X)$ with $p^\ast\xi=0$. For such $\xi$ we can define a height function $h:\overline{X}\to\R$ and neighborhoods $N\subset\overline{X}$ of infinity with respect to $\xi$ as in the case of the universal covering.

\begin{definition}
 A homology class $z\in H_q(\overline{X};R)$ is said to be \em movable to infinity in $\overline{X}$ \em with respect to $\xi\in S(X,\overline{X})$, if $z$ can be realized by a singular cycle in any neighborhood $N$ of infinity with respect to $\xi$.
\end{definition}

In other words, $z\in H_q(\overline{X};R)$ is movable to infinity with respect to $\xi$, if $z$ is an element of
\[ \bigcap_N {\rm Im}(H_q(N;R)\to H_q(\overline{X};R)) \]
where the intersection is taken over all neighborhoods of infinity with respect to $\xi$.

Note that we have an inverse system of $R$-modules given by $\{H_q(\overline{X},N;R)\leftarrow H_q(\overline{X},N';R)\}$ which runs over neighborhoods $N'\subset N$ of infinity with respect to $\xi$. So $z$ is movable to infinity with respect to $\xi$ if and only if
\[ z \in \ker(H_q(\overline{X};R)\to \liminv H_q(\overline{X},N;R)). \]
Just as with integer coefficients, we can define a Novikov-Sikorav ring $\widehat{RG}_\xi$ with coefficients in an arbitrary ring $R$. As this ring can be expressed as an inverse limit, standard methods give the following exact sequence, see \cite{farscm} or \cite{geoghe} for details.
\begin{equation}\label{homlimseq}
 0\rightarrow \limone H_{q+1}(\overline{X},N;R)\rightarrow H_q(X;\widehat{RG}_{-\xi})\rightarrow \liminv H_q(\overline{X},N;R)\rightarrow 0
\end{equation}
Let us give a simple criterion for $z\in H_i(\overline{X};R)$ to be movable to infinity.

\begin{definition}
 An element $\Delta\in RG$ is said to have \em $\xi$-lowest coefficient 1\em, if $\Delta=1-y$, with $y=\sum a_jg_j\in RG$ and such that each $g_j\in G$ satisfies $\xi(g_j)>0$.
\end{definition}
Such a $\Delta$ is invertible over $\widehat{RG}_{-\xi}$. So if $\Delta\cdot z=0 \in H_q(\overline{X};R)$, we get that the image of $z$ in $H_q(X;\widehat{RG}_{-\xi})$ is zero, and by (\ref{homlimseq}), $z$ is movable to infinity.

Under certain conditions, this is in fact necessary for movability to infinity. The following theorem is taken from \cite{farscm,farsce}.

\begin{theorem}\label{movcrit}
 Let $X$ be a finite connected CW-complex and $p:\overline{X}\to X$ a regular covering with covering transformation group $G\cong \Z^r$. Let $\xi\in S(X,\overline{X})$ induce an injective homomorphism $G\to\R$. Let $R$ be either $\Z$ or a field. For $z\in H_q(\overline{X};R)$, the following are equivalent:
\begin{enumerate}
 \item $z$ is movable to infinity with respect to $\xi$.
 \item $i_\ast(z)=0\in H_q(X;\widehat{RG}_{-\xi})$, where $i_\ast:H_q(\overline{X};R)\cong H_q(X;RG)\to H_q(X;\widehat{RG}_{-\xi})$ is change of coefficients.
 \item There is $\Delta\in RG$ with $\xi$-lowest coefficient 1 such that $\Delta\cdot z=0$.
\end{enumerate}
\end{theorem}

In the case that $R$ is a field, condition (3) is equivalent to the existence of a non-zero $\Delta\in RG$ with $\Delta\cdot z=0\in H_q(\overline{X};R)$, as we can find $r\in R$ and $g\in G$ such that $\Delta\cdot rg$ has $\xi$-lowest coefficient 1.

Notice that $\Delta\in RG$ having $\xi$-lowest coefficient 1 is an open condition in $\xi\in S(X,\overline{X})$; in particular, if $z$ is movable to infinity with respect to $\xi$, it is also movable to infinity with respect to nearby $\xi'$, under the conditions of Theorem \ref{movcrit}.

The equivalence of (1) and (2) is obtained by showing that the $\limone$-term in (\ref{homlimseq}) vanishes. This is done in \cite{farscm,farsce} under the conditions of Theorem \ref{movcrit}. Using Usher \cite[Thm.1.3]{usher} one can see that the $\limone$ term in (\ref{homlimseq}) vanishes in fact for any abelian covering and any Noetherian ring $R$.

Theorem \ref{movcrit} is an important ingredient in obtaining lower bounds for $\cat(X,\xi)$ and $\cat^1(X,\xi)$ via cup-lengths. We refer the reader to \cite{farsch,farsce} for details.

\section{Function spaces of paths to infinity}\label{latour}

If $\gamma:[0,\infty)\to X$ is a map, we can lift it to a map $\tilde{\gamma}:[0,\infty)\to \tilde{X}$ into the universal cover. We want to look at those maps, which lift to paths to $\infty$ with respect to a given $\xi$, compare Definition \ref{pathtoinfinity}. This does not depend on the particular lift of $\gamma$.

We set
\[
\mathcal{C}_\xi(X)\,=\,\{\gamma:[0,\infty)\to X\,|\,\tilde{\gamma} \mbox{ is a path to infinity with respect to }\xi\}.
\]
It is equipped with the following topology: For $a,b\in[0,\infty)$ and $U$ open in $X$ let
\begin{eqnarray*}
W(a,b;U)&=&\{\gamma\in\mathcal{C}_\xi(X)\,|\,\gamma([a,b])\subset U\}
\end{eqnarray*}
and for $a,A\in[0,\infty)$ let
\begin{eqnarray*}
W(a,A)&=&\{\gamma\in\mathcal{C}_\xi(X)\,|\,\forall t\geq a\,\,\,h_\xi\tilde{\gamma}(t)-h_\xi\tilde{\gamma}(0)>A\}.
\end{eqnarray*}
These sets form a subbasis for the topology of $\mathcal{C}_\xi(X)$. Notice that the sets $W(a,b;U)$ provide the compact-open topology on $\mathcal{C}_\xi(X)$ while the sets $W(a,A)$ give a "control at infinity".

The evaluation $e:\mathcal{C}_\xi(X)\to X$ given by $e(\gamma)=\gamma(0)$ is a fibration and for $x_0\in X$ we have the fiber
\begin{eqnarray*}
\mathcal{M}_\xi&=&\{\gamma\in\mathcal{C}_\xi(X)\,|\,\gamma(0)=x_0\},
\end{eqnarray*}
compare \cite{latour}.

\begin{remark}
 If we consider $\mathcal{M}_\xi$ with the compact-open topology, we get that $\mathcal{M}_\xi$ is contractible. To see this, choose a $\gamma_\infty\in \mathcal{M}_\xi$. For any $\gamma\in \mathcal{M}_\xi$ and $t\in [0,\infty)$, let $\gamma_t\in \mathcal{M}_\xi$ be given by
\begin{eqnarray*}
 \gamma_t(s)&=&\left\{\begin{array}{ll}\gamma_\infty(s)& 0\leq s\leq t \\ \gamma_\infty(2t-s)&t\leq s\leq 2t \\ \gamma(s-2t) & 2t\leq s \end{array} \right.
\end{eqnarray*}
It is easy to see that $H:\mathcal{M}_\xi\times [0,\infty]\to \mathcal{M}_\xi$ given by $H(\gamma,t)=\gamma_t$ is continuous in the compact-open topology, and hence defines a contraction. Of course $H$ is no longer continuous with the $W(a,A)$ open in $\mathcal{M}_\xi$.
\end{remark}

\begin{remark}
 Let us give another interpretation of the topology on $\mathcal{M}_\xi$. Let $\tilde{X}_\infty=\tilde{X}\cup\{\infty\}$, that is, we add a point $\infty$ to $\tilde{X}$; the topology on $\tilde{X}_\infty$ is generated by the open sets of $\tilde{X}$ and sets $N\cup \{\infty\}$, where $N$ is an open neighborhood of $\infty$ with respect to $\xi$ (in the sense of Section \ref{sigmacw}).

Write
\begin{eqnarray*}
 \mathcal{P}(\tilde{X}_\infty)&=&\{\gamma:[0,\infty]\to \tilde{X}_\infty\,|\,\gamma(\infty)=\infty\},
\end{eqnarray*}
which we topologize with the compact open topology. This is a usual path space with $\infty$ as the basepoint.
The space $\mathcal{C}_\xi$ of paths to infinity with respect to $\xi$ can be identified with the subspace of $\mathcal{P}(\tilde{X}_\infty)$ consisting of those $\gamma$ with $\gamma([0,\infty))\subset \tilde{X}$. Similarly, $\mathcal{M}_\xi$ can be identified with a subspace of $\Omega(\tilde{X}_\infty)=e^{-1}(\{\tilde{x_0}\})$, where $\tilde{x_0}\in p^{-1}(\{x_0\})\subset\tilde{X}$ and $e:\mathcal{P}(\tilde{X}_\infty)\to \tilde{X}_\infty$ the usual fibration. Note that $\mathcal{C}_\xi$ is a covering space of $\mathcal{C}_\xi(X)$ with covering group $\pi_1(X)$.
\end{remark}

Given $\gamma_0\in \mathcal{M}_\xi$, we want to examine the homotopy groups $\pi_k(\mathcal{M}_\xi,\gamma_0)$ for $k\geq 0$. For this let $g:(S^k,\ast)\to(\mathcal{M}_\xi,\gamma_0)$ be a map. It gives rise to a map $\phi_g:S^k\times [0,\infty)\to X$ by $\phi_g(x,t)=g(x)(t)$ and since $\phi_g(x,0)=x_0$, a map $\Phi_g:\R^{k+1}\to X$ such that $\Phi_g(x\cdot t)=\phi(x,t)$ for $x\in S^k\subset\R^{k+1}$. If we assume that $\ast\in S^k$ corresponds to $(1,0,\ldots,0)\in\R^{k+1}$, we have $\Phi_g(t,0,\ldots,0)=\gamma_0(t)$. Furthermore, if we lift $\Phi_g$ to a map $\tilde{\Phi}_g:\R^{k+1}\to\tilde{X}$, we get that $h_\xi\circ\tilde{\Phi}(x)\to\infty$, as $|x|\to\infty$. A homotopy between two maps $g_0,g_1:(S^k,\ast)\to(\mathcal{M}_\xi,\gamma_0)$ relative to the basepoint corresponds to a homotopy $\Phi:\R^{k+1}\times [0,1]\to X$ between $\Phi_{g_0}$ and $\Phi_{g_1}$ relative to $[0,\infty)\times \{0\}$ and such that $h_\xi\circ\tilde{\Phi}(x,s)\to\infty$ as $|x|\to\infty$ uniformly in $s\in[0,1]$, for a lifting $\tilde{\phi}$.

Now assume we have a sequence $(N_i)_{i\geq 0}$ of neighborhoods of $\infty$ with respect to $\xi$ such that $N_i\subset N_{i-1}$ for all $i$ and $\bigcap_{i\geq 0} N_i=\emptyset$. Let $\tilde{x}_0\in\tilde{X}$ be a lifting of $x_0\in X$. Also, let $\tilde{\gamma}_0:[0,\infty)\to \tilde{X}$ be the lifting of $\gamma_0$ with $\tilde{\gamma}_0(0)=\tilde{x}_0$. Pick a sequence $t_i>0$ such that $t_{i+1}>t_i$ for all $i\geq 0$ such that $\tilde{\gamma}_0(t)\in N_i$ for all $t\geq t_i$. This sequence exists by the definition of $\gamma_0\in\mathcal{M}_\xi$. We let
\begin{eqnarray*}
y_i&=&\tilde{\gamma}_0(t_i)\in N_i
\end{eqnarray*}
be the sequence of basepoints of $N_i$ for all $i\geq 0$. We get a natural homomorphism $\chi_i:\pi_k(\tilde{X},N_{i+1},y_{i+1})\to \pi_k(\tilde{X},N_i, y_i)$ induced by inclusion where the change of basepoint is done using the path $\tilde{\gamma}_0|[t_i,t_{i+1}]$. This gives rise to an inverse system $(\chi_i:\pi_k(\tilde{X},N_{i+1},y_{i+1})\to \pi_k(\tilde{X},N_i,y_i))$ for every $k\geq 0$. Note that this is an inverse system of pointed sets for $k\leq 1$.

For every $i\geq 0$ and $k\geq 0$ define $\varphi_i:\pi_k(\mathcal{M}_\xi,\gamma_0)\to\pi_{k+1}(\tilde{X},N_i,y_i)$ in the following way. Given $g:(S^k,\ast)\to (\mathcal{M}_\xi,\gamma_0)$, we define $\varphi_g:(D^{k+1},S^k,\ast)\to(\tilde{X},N_i,y_i)$ by $\varphi_g(x\cdot t)=\tilde{\Phi}_g(\lambda(x)\cdot x\cdot t)$ for $x\in S^k$ and $t\in[0,1]$, where $\lambda:S^k\to (0,\infty)$ is a map such that $\lambda(\ast)=t_i$, and for every $x\in S^k$ we have $g(x)(t)\in N_i$ for all $t\geq \lambda(x)$. It is clear that the homotopy class of $\varphi_g$ does not depend on the particular choice of $\lambda$. Basically we use the map $\tilde{\Phi}_g$ and restrict it to a large enough ball in $\R^{k+1}$.

This induces the map $\varphi_i:\pi_k(\mathcal{M}_\xi,\gamma_0)\to\pi_{k+1}(\tilde{X},N_i,y_i)$ and we clearly have $\chi_i\varphi_{i+1}=\varphi_i$. Thus we get an induced map
\begin{eqnarray*}
\varphi:\pi_k(\mathcal{M}_\xi,\gamma_0)&\longrightarrow & \liminv \,\pi_{k+1}(\tilde{X}, N_i,y_i)
\end{eqnarray*}
which is a group homomorphism for $k\geq 1$ and a map of pointed sets for $k=0$.

Next we define a map $\psi':\prod_{i\geq 0}\pi_{k+2}(\tilde{X},N_i,y_i)\to \pi_k(\mathcal{M}_\xi,\gamma_0)$, so let $a_i\in\pi_{k+2}(\tilde{X},N_i,y_i)$. Represent $a_i$ by a map $g_i:(D^{k+2},S^{k+1},\ast)\to(\tilde{X},N_i,y_i)$. Furthermore, let $G':\R^{k+1}\to \tilde{X}$ be given by $G'(x\cdot t)=\tilde{\gamma}_0(t)$ for $x\in S^k$ and $t\in[0,\infty)$. Let $B(t_i)$ be a small disc with center at $(-t_i,0)\in \R^{k+1}$ such that the $B(t_i)$ are pairwise disjoint. We use $(-t_i,0)$ as the center since we want to change $G'$ on the $B(t_i)$ without changing it on $[0,\infty)\times\{0\}\subset\R^{k+1}$. So homotop $G'$ to a map $G$ relative to $\R^{k+1}-\bigcup B'(t_i)$ such that $G$ is constant to $\tilde{\gamma}_0(t_i)$ on $B(t_i)$ for all $i\geq 0$. Here $B'(t_i)$ is a slightly bigger disc such that they are still pairwise disjoint.

If we restrict the map $g_i$ to $S^{k+1}$, we can think of this map as a map $\tilde{g}_i:(D^{k+1},S^k)\to (N_i,y_i)$. Now we can replace $G$ by a map $\tilde{G}:\R^{k+1}\to\tilde{X}$ such that $\tilde{G}|B(t_i)\equiv \tilde{g}_i$ and $\tilde{G}$ agrees with $G$ everywhere else. Clearly $\tilde{G}$ induces a map $g:(S^k,\ast)\to(\mathcal{M}_\xi,\gamma_0)$ and this defines a map
\begin{eqnarray*}
\psi':\prod_{i\geq 0}\pi_{k+2}(\tilde{X},N_i,y_i)&\longrightarrow & \pi_k(\mathcal{M}_\xi,\gamma_0).
\end{eqnarray*}
Note here that if $g'_i:(D^{k+2},S^{k+1},\ast)\to(\tilde{X},N_i,y_i)$ also represents $a_i$ for every $i\geq 0$, we get that $g_i|S^{k+1}$ is homotopic to $g'_i|S^{k+1}$ within $N_i$, so the resulting maps $g$ and $g'$ represent the same element in $\pi_k(\mathcal{M}_\xi,\gamma_0)$.
It is worth pointing out that $\tilde{G}$ is homotopic to $G$ relative to $[0,\infty)\times\{0\}$, but the resulting function $H:S^k\times [0,1]\to\mathcal{M}_\xi$ need not be continuous, since the maps $g_i$ are only null homotopic in $\tilde{X}$, but not necessarily with control.

Let us recall the definition of the derived limit for our inverse system. Two sequences $(a_i),(b_i)\in\prod_{i\geq 0}\pi_{k+2}(\tilde{X},N_i,y_i)$ are called equivalent, if there exists a sequence $(c_i)\in\prod_{i\geq 0}\pi_{k+2}(\tilde{X},N_i,y_i)$ such that $b_i=c_i\cdot a_i\cdot \chi_i(c_{i+1})$ for all $i\geq 0$. Then $\limone\,\pi_{k+2}(\tilde{X},N_i,y_i)$ is the set of equivalence classes. For $k\geq 1$ this has the structure of an abelian group, but for $k=0$ we only get a pointed set.

It is easy to see that $\psi'$ induces a map
\begin{eqnarray*}
\psi:\limone\,\pi_{k+2}(\tilde{X},N_i,y_i)&\longrightarrow&\pi_k(\mathcal{M}_\xi,\gamma_0)
\end{eqnarray*}
which is a homomorphism for $k\geq 1$ and a map of pointed sets for $k=0$.

\begin{proposition}\label{exactseq}
With the above notation there is a short exact sequence
\[ 1\to \limone\,\pi_{k+2}(\tilde{X},N_i,y_i)\stackrel{\psi}{\longrightarrow} \pi_k(\mathcal{M}_\xi,\gamma_0)\stackrel{\varphi}{\longrightarrow} \liminv\,\pi_{k+1}(\tilde{X},N_i,y_i)\to 1\]
which is a short exact sequence of groups in the case $k\geq 1$ and of pointed sets in the case $k=0$. If $k=0$, $\psi$ is also injective.
\end{proposition}

The proof is standard and will be omitted, see also \cite{geoghe}. Notice the similarity between this sequence and (\ref{homlimseq}).

We can use this to give another equivalent definition for $\xi\in \Sigma^k(X)$.

\begin{proposition}
\label{anotherequdef}
Let $X$ be a finite connected CW-complex, $\xi\in H^1(X;\R)$ be nonzero and $k\geq 1$. Then $\xi\in \Sigma^k(X)$ if and only if $\mathcal{M}_\xi$ is $(k-1)$-connected.
\end{proposition}

\begin{proof}
If $\mathcal{M}_\xi$ is $(k-1)$-connected, then by Proposition \ref{exactseq} we get that the inverse system $\{\pi_l(\tilde{X},N)\}$ is pro-trivial for $l\leq k$ which gives $\xi\in \Sigma^k(X)$ by Proposition \ref{equivdef}.

To get the other direction, we have to worry about $\limone\,\pi_{k+1}(\tilde{X},N_i)$. But by the next lemma $\{\pi_{k+1}(\tilde{X},N)\}$ is semi-stable, so the $\limone$-term vanishes. It follows from Proposition \ref{exactseq} that $\mathcal{M}_\xi$ is $(k-1)$-connected.
\end{proof}

\begin{lemma}
Let $X$ be a finite connected CW-complex, $\xi\in H^1(X;\R)$ be nonzero and $h_\xi:\tilde{X}\to\R$ a height function. If $\xi\in \Sigma^k(X)$, then $\{\pi_{k+1}(\tilde{X},N)\}$ is semi-stable.
\end{lemma}

\begin{proof}
Use the homotopy from Proposition \ref{equivdef}(4) to push any $k$-sphere in $N$ arbitrarily far away, this gives semi-stability as the pushing may be done in $N'$ slightly bigger than $N$ (depends on $H$ only).
\end{proof}

\begin{appendix}

\section{Sigma invariants of chain complexes}\label{appendicitis}

In this appendix we show how $\Sigma^k(C)$ is related to criteria involving chain homotopies on $C$.

If the chain complex $C$ consists of flat $R$-modules, we have the following criterion.

\begin{proposition}\label{critone}
Let $C$ be a chain complex of flat $R$-modules and $n$ a non-negative integer. Then the following are equivalent.
\begin{enumerate}
\item $C$ is of finite $n$-type.
\item For every index set $J$, the natural map $H_k(C,\prod_J R)\to \prod_J H_k(C)$ is an isomorphism for $k<n$ and an epimorphism for $k=n$.
\end{enumerate}
\end{proposition}

This is basically \cite[Thm.2]{brown1}, but we allow $C$ to be flat and not necessarily projective so we only have a homology criterion. The proof goes through for flat modules.

By a filtration of $C$ we mean a family $\{C^\alpha\}_{\alpha\in\mathcal{A}}$ of sub-chain complexes where $\mathcal{A}$ is a directed set, $C^\alpha\subset C^\beta$ for $\alpha\leq \beta$ and $C=\bigcup C^\alpha$.

Given a filtration, we define $D^\alpha=C/C^\alpha$.

The analogue of \cite[Thm.2.2]{brown2} is

\begin{theorem}\label{browncrit}
Let $C$ be an $R$-chain complex with a filtration $\{C^\alpha\}_{\alpha\in\mathcal{A}}$ of finite $n$-type complexes $C^\alpha$ of flat $R$-modules. Then $C$ is of finite $n$-type if and only if the direct system $\{H_i(D^\alpha)\}$ is essentially trivial for $i\leq n$.
\end{theorem}

\begin{proof}
The proof is similar to the proof of \cite[Thm.2.2]{brown2}, but one has to use Proposition \ref{critone}. We omit the details.
\end{proof}

Let $G$ be a finitely generated group and $\xi:G\to\R$ a non-zero homomorphism. Let $C$ be a free $\Z G$-chain complex which is finitely generated in every degree. Given a valuation on $C$ extending $\xi$, we can define a subcomplex
\begin{eqnarray*}
C^\xi&=&\{x\in C\,|\,v(x)\geq 0\}
\end{eqnarray*}
Valuations are determined by their value on basis elements, so it is easy to see that $C^\xi$ is a finitely generated free chain complex over $\Z G_\xi$, compare \cite[Lemma 3.1]{bieren}. Given $g\in G$, we can also look at the subcomplex $gC^\xi\subset C$ which is isomorphic to $C^\xi$. Denote $D^g=C/gC^\xi$.

We use Theorem \ref{browncrit} to get

\begin{proposition}
Let $C$ be a free $\Z G$-chain complex which is finitely generated in every degree, $v$ a valuation extending $\xi$ and $n$ a non-negative integer. Then the following are equivalent.
\begin{enumerate}
\item $\xi\in \Sigma^n(C)$.
\item The direct system $\{H_i(D^g)\}$ is essentially trivial for $i\leq n$.
\item There exists a chain map $\varphi:C\to C$ chain homotopic to the identity such that $v(\varphi(x))>v(x)$ for all non-zero $x\in C_i$ with $i\leq n$.\qed
\end{enumerate}
\end{proposition}

\end{appendix}


\begin{thebibliography}{99}

\bibliographystyle{amsalpha}
%\begin{thebibliography}{99}







\bibitem{besbra}M. Bestvina, N. Brady, \textit{Morse theory and finiteness properties of groups}, Invent.\ Math.\ 129 (1997), 445-470.
\bibitem{bieri}R. Bieri, \textit{Finiteness length and connectivity length for groups}, in: Geometric group theory down under (Canberra, 1996),  9-22, de Gruyter, Berlin, 1999.
\bibitem{biegeo}R. Bieri, R. Geoghegan, \textit{Connectivity properties of group actions on non-positively curved spaces}, Mem. Amer. Math. Soc. 161 (2003),  no. 765.
\bibitem{bigeko}R. Bieri, R. Geoghegan, D. Kochloukova, \textit{The Sigma invariants of Thompson's group $F$}, preprint.
\bibitem{binest}R. Bieri, W. Neumann, R. Strebel, \textit{A geometric invariant of discrete groups}, Invent. Math. 90 (1987), 451-477.
\bibitem{bieren}R. Bieri, B. Renz, \textit{Valuations on free resolutions and higher geometric invariants of groups}, Comment.\ Math.\ Helv.\ 63 (1988), 464-497.
\bibitem{brown1}K. Brown, \textit{Homological criteria for finiteness},  Comment.\ Math.\ Helv.\ 50 (1975), 129-135.
\bibitem{brown}K. Brown, \textit{Cohomology of groups}, Springer-Verlag, New York-Berlin, 1982.
\bibitem{brown2}K. Brown, \textit{Finiteness properties of groups}, J. Pure Appl.\ Algebra 44 (1987), 45-75.
%\bibitem{DNF} B. Dubrovin, S. Novikov and A. Fomenko,
%\textit{ Modern Geometry; Methods of the homology theory}, 1984.
\bibitem{farber} M. Farber, \textit{Zeros of closed 1-forms, homoclinic orbits and Lusternik-Schnirelman theory},  Topol.\ Methods Nonlinear Anal.\  19 (2002), 123-152.
\bibitem{farbe4} M. Farber, \textit{Lusternik-Schnirelman theory and dynamics}, "Lusternik-Schnirelmann Category and Related Topics", Contemporary Mathematics, ~ 316(2002), 95 - 111.
\bibitem{farbook} M. Farber, \textit{Topology of closed one-forms}, Mathematical Surveys and Monographs, AMS, 2004.
\bibitem{farkap} M. Farber, T. Kappeler, \textit{Lusternik-Schnirelman theory and dynamics, II},  Proc. Steklov Inst. Math.  2004,  no. 4 (247), 232--245.
%\bibitem{farzeros} M. Farber, D. Sch\"utz, \textit{Closed 1-forms with at most one zero}, Topology, 45(2006), 465-473.
\bibitem{farscm}M. Farber, D. Sch\"utz, \textit{Moving homology classes to infinity}, Forum Math. 19 (2007), 281-296.
\bibitem{farsch}M. Farber, D. Sch\"utz, \textit{Cohomological estimates for $cat(X,\xi)$}, Geom.\ Top. 11 (2007), 1255-1289.
\bibitem{farsce}M. Farber, D. Sch\"utz, \textit{Homological category weights and estimates for $\cat^1(X,\xi)$}, J. Eur. Math. Soc. 10 (2008), 243-266.
\bibitem{fasc}M. Farber, D. Sch\"utz, \textit{Closed 1-forms in topology and dynamics}, preprint.
\bibitem{geoghe}R. Geoghegan, \textit{Topological methods in group theory}, Springer-Verlag, New York, 2008.
\bibitem{harkoc}J. Harlander, D. Kochloukova, \textit{The $\Sigma^3$-conjecture for metabelian groups},  J. London Math. Soc. (2)  67 (2003), 609-625.
\bibitem{latour}F. Latour, \textit{Existence de 1-formes ferm\'ees non singuli\`eres dans une classe de cohomologie de de Rham}, Publ. IHES 80 (1994), 135-194.
%\bibitem{lynsch}R. Lyndon, P. Schupp, \textit{Combinatorial group theory}. Springer-Verlag, Berlin-New York, 1977.


%\bibitem{Levitt}
%G. Levitt, \textit{1-formes ferm\'ees singuli\`eres et
%groupe fondamental}, Invent. Math., \textbf{88}(1987), 635 - 667.
\bibitem{memewy}J. Meier, H. Meinert, L. van Wyk, \textit{Higher generation subgroup sets and the $\Sigma$-invariants of graph groups}, Comment.\ Math.\ Helv.\ 73 (1998), 22-44.

%\bibitem{mislin}G. Mislin, \textit{Conditions for finite domination for certain complexes}. Algebraic topology (Proc. Conf., Univ. British Columbia, Vancouver, B.C., 1977),  pp. 219-224, Lecture Notes in Math. 673, Springer, Berlin, 1978.

\bibitem {N1}  S. P. Novikov,
\textit{ Multi-valued functions and functionals. An analogue of Morse theory},
 Soviet Math. Doklady, \textbf{24}(1981),  pp. 222--226.

\bibitem{noviko}S.P. Novikov, \textit{The Hamiltonian formalism and a multi-valued analogue of Morse theory}, Russian Math. Surveys 37 (1982), 1-56.

%\bibitem {N3}
% S. P. Novikov,
%\textit{Bloch homology, critical points of functions and closed 1-forms},
% Soviet Math. Dokl., \textbf{33}( 1986), pp. 551--555.

\bibitem{N4} S. P. Novikov, \textit{Variational methods and periodic solutions of equations
of Kirchhoff type, II}, Functional Analysis and Its Applications,
\textbf{15}(1981), pp. 263 - 274.

%\bibitem{N5} S. P. Novikov, \textit{Quasiperiodic structures in topology},
%Topological Methods
%in Modern Mathematics, Publish or Perish, 1993, pp. 223 - 235.

\bibitem{NSm} S. P. Novikov, I. Smel'tser, \textit{Periodic solutions of the
Kirhhoff equations for the free motions of a rigid body in a liquid, and
extended Lusternik-Schnirelman-Morse theory, I} Functional Analysis and Its
Applications, \textbf{15}(1981), pp. 197 - 207.

%\bibitem {NSu} S. Novikov, M. Shubin,
%\textit{ Morse inequalities and von Neumann $II_1$-factors},
% Soviet Math. Dokl., \textbf{34}(1987), pp. 79 - 82.

%\bibitem{NT} S. P. Novikov, I. Taimanov,
%\textit{Periodic extremals in multivalued or not everywhere positive
%functionals}, Soviet Math. Doklady, \textbf{29}(1984), pp. 18 - 20.



\bibitem{ranims}A. Ranicki, \textit{The algebraic theory of finiteness obstruction}, Math.\ Scand.\ 57 (1985), 105-126.
%\bibitem{schman}D. Sch\"utz, \textit{On the Lusternik-Schnirelman theory of a real cohomology class}, Manu-scripta Math. 113 (2004), 85-106.
\bibitem{schfin}D. Sch\"utz, \textit{Finite domination, Novikov homology and nonsingular closed 1-forms}, Math.\ Z.  252 (2006), 623-654.
\bibitem{sikora}J.-Cl. Sikorav, \textit{Th\`ese}, Universit\'e Paris-Sud, 1987.

%\bibitem{Sik} J.-Cl. Sikorav, \textit{Un probl\`eme de disjonction par isotopie
%symplectique dans un fibr\'e cotangent}, Ann. Sci. \'Ecole Norm. Sup.
%(4)\textbf{19}(1986), 543-552.



\bibitem{spanie}E.H. Spanier, \textit{Algebraic topology}. McGraw-Hill Book Co., New York-Toronto, Ont.-London 1966.

%\bibitem {Ta} F. Takens,
%\textit{ The minimal number of critical points of a function on a compact
%manifold and the Lusternik - Schnirelmann category}, Invent. Math.,
%\textbf{6}(1968), pp. 197 - 244.


\bibitem{usher}M. Usher, \textit{Spectral numbers in Floer theories}, preprint, available as \verb+arXiv:0709.1127+
\bibitem{wall}C.T.C. Wall, \textit{Finiteness conditions for ${\rm CW}$-complexes},
Ann.\ of Math.\ (2) 81 (1965), 56-69.
\bibitem{wall2}C.T.C. Wall, \textit{Finiteness conditions for CW-complexes II}, Proc.\ Roy.\ Soc.\ Ser.\ A 295 (1966), 129-139.
\end{thebibliography}
\end{document}